\documentclass[11pt,letterpaper]{amsart}
\usepackage[margin=1.0in]{geometry}
\usepackage[english]{babel}
\usepackage[utf8]{inputenc}
\usepackage{ulem}
\usepackage{amsmath,amssymb,amsthm, comment, mathtools}
\usepackage{hyperref,mathrsfs,xcolor}
\usepackage{enumitem}
\usepackage{genealogytree}

\newcommand{\R}{\mathbb{R}}

\newcommand{\ve}{\varepsilon}

\usepackage[mathscr]{euscript}

\newcommand{\eee}{equation}
\newcommand{\be}{\begin{\eee}}
\newcommand{\ee}{\end{\eee}}

\numberwithin{equation}{section}
\newtheorem{lemma}{Lemma}[section]
\newtheorem{prop}[lemma]{Proposition}
\newtheorem{theorem}[lemma]{Theorem}
\newtheorem{cor}[lemma]{Corollary}

\theoremstyle{definition}
\newtheorem{remark}[lemma]{Remark}
\newtheorem{defi}[lemma]{Definition}

\newtheorem{example}[lemma]{Example}

\mathtoolsset{showonlyrefs,showmanualtags}

\title{Pensive billiards,  point vortices, and the silver ratio}

\author[T. D. Drivas]{Theodore D. Drivas}
\address{Department of Mathematics, Stony Brook University, Stony Brook, NY, 11790}
\email{tdrivas@math.stonybrook.edu}

\author[D.  Glukhovskiy]{Daniil  Glukhovskiy}
\address{Department of Mathematics, Stony Brook University, Stony Brook, NY, 11790}
\email{daniil.glukhovskiy@stonybrook.edu }

\author[B.  Khesin]{Boris Khesin}
\address{Department of Mathematics, University of Toronto, ON M5S 2E4, Canada}
\email{khesin@math.toronto.edu}

\begin{document}

\maketitle
\vspace{-3mm}
\begin{abstract}
\vspace{-6mm}
We define a new class of plane billiards -- the `pensive billiard' -- in which  the billiard ball travels along the boundary for some distance depending on the incidence angle before reflecting, while preserving the billiard rule of equality of the angles of incidence and reflection.   This generalizes so called `puck billiards' proposed by M.~Bialy, as well as a `vortex billiard', i.e. the motion of a point vortex dipole in 2D hydrodynamics on domains with boundary. We prove the variational origin and invariance of a symplectic structure for pensive billiards, as well as study their properties including conditions for a twist map, the existence of periodic orbits, etc. We also demonstrate the appearance of both the golden and silver ratios in the corresponding hydrodynamical vortex setting. Finally, we introduce and describe basic properties of pensive outer billiards.
\end{abstract}
\vspace{-7mm}

\tableofcontents


\section{Introduction}
In this paper we introduce a new discrete dynamical system which generalizes classical billiards. Given a planar domain, \textit{pensive billiard map} is a composition of the classical billiard map with the translation along the boundary, where this translation may depend on the incidence angle (but not the point of incidence), see Figure \ref{figPensive1}. The classical billiard is recovered if the translation is identically zero; however the dependence can be chosen arbitrarily, thus allowing for description of several physical systems as particular cases of pensive billiards. 

\begin{figure}[h!]\centering
    \includegraphics[width=.45\columnwidth]{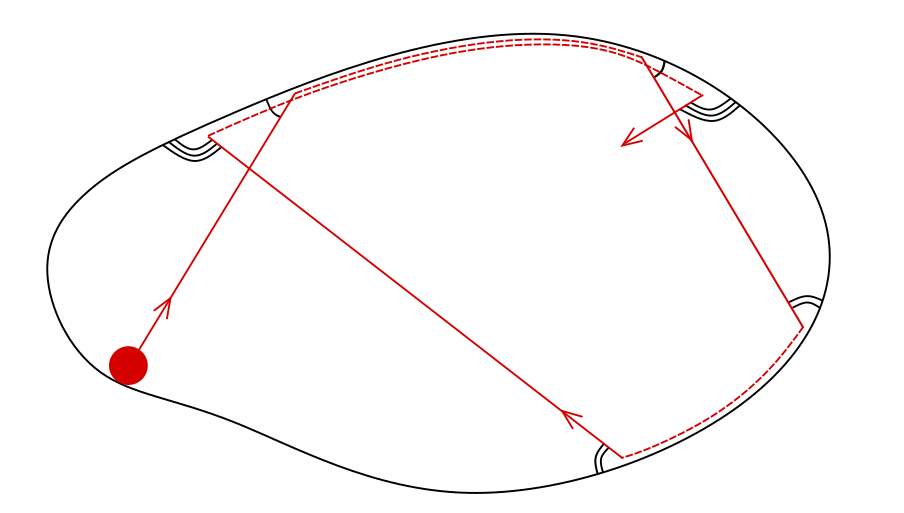} 
  \caption{Pensive Billiard.}
\label{figPensive1}
\end{figure}

One of the systems that admits such a description is a \textit{puck billiard} proposed by M.~Bialy in \cite{Bialy}. Namely,
consider a surface homeomorphic to the sphere obtained by gluing two copies of a given domain $D$ to the cylinder $[0,h] \times \partial D$ (see Figure~\ref{figpuck}). Geodesics on this surface can be identified with  pensive billiard trajectories for a certain delay function, as we discuss below. 

Another example, which we study in detail in this paper, arises in studying 2-dimensional ideal hydrodynamics. A point vortex model describes a motion of a fluid with sharply concentrated vorticity. It is known that, in the limit of zero separation, a pair of oppositely rotating vortices in a planar domain travels together in the straight line perpendicular to the segment connecting them. Once they hit the boundary at a certain angle, they split and travel along the boundary in the opposite directions. Then, if the boundary is closed, they will meet again and merge into a traveling dipole, reflecting from the boundary at the same angle, while their position has shifted (see Figure~\ref{figVortexBilliard}). This motion can be described by a pensive billiard with another explicitly computed delay function, see Section~\ref{sect:vortex-billiard} for more details.

There is an abundance of literature on conventional dynamical billiards; see \cite{Tab}. In this paper we initiate a study of pensive billiards, specifically focusing on vortex systems. In particular we extend several key billiard results to this general setting. The paper is organized as follows:

In Section \ref{sect:pensive} we examine pensive billiards in generality, and prove analogs of classical results. We show area-preservation, derive its variational principle and generating function, as well as discuss additional assumptions under which the twist condition holds. We also provide examples of dynamical behavior.

In Section \ref{sect:puck-billiard}, we generalize puck billiards and prove the correspondence of geodesics on a wider class of surfaces to pensive billiards.

In Section \ref{sect:vortex-billiard}, we study the motion of point vortices and its corresponding pensive billiard systems. Some results of this section are of independent interest for fluid dynamics: in particular we observe the appearance of the silver ratio in hydrodynamical problems (supplementing the appearance of the golden ratio
encountered in \cite{KhesinWang}).

In Section \ref{sect:outer} we define \textit{pensive outer billiards} and prove projective duality of them to pensive billiards on the sphere.


\section{Pensive billiards}\label{sect:pensive}

\subsection{Definition of pensive billiards}

We begin with a planar domain $D\subset \R^2$ whose boundary $\gamma = \partial D$ is a smooth closed curve. The phase space of a pensive billiard system is given by
\be
M=\{(q,v) ~|~q\in \gamma \text{ and } v\in T_q\R^2 \text{ an inward pointing unit vector}\}.
\ee
 It  represents the location and direction of motion of a billiard ball right after impact with the boundary. The classical billiard ball map $\mathsf{CB}\colon M\to M$ sends $(q_1,v_1)\mapsto (q_2,v_2)$, as obtained by traveling the table from $q_1$ in direction $v_1$ until one hits $\gamma$ at $q_2$, and then reflecting the velocity vector $v_1$ about $T_{q_2}\gamma$ to obtain $v_2$ according to the usual reflection law. 

There are several ways to introduce coordinates in $M$. 
Parametrize $\gamma$ by arc length as $\gamma(s)$ for $s\in [0,2L]$, and, given $v\in T_{\gamma(s)}\R^2$, denote by $\theta\in (0,\pi)$ the angle between $v$ and $\gamma'(s)$. 
Thus $M$ is a cylinder with coordinates $(s,\theta)$.
Alternatively,  let $p={\rm proj} \,v|_{T_q\gamma}$ be the projection of $v$ to the tangent space to $\gamma$ at the point $q$. (Note that $p:=\cos\theta$.)
Then $M$ can be identified with the (co)tangent bundle of unit balls,  
\be
M=T^{(*)}_{< 1}\gamma=\{(q, p)~|~q\in \gamma, \ -1 <p<1\}
\ee
 (the Euclidean structure in $\R^2$ allows one to identify tangent and cotangent spaces). 
 Recall that $M$ is a symplectic manifold: Let $\lambda=p\,dq$ and $\omega=d\lambda$ be the Liouville 1-form and the corresponding standard symplectic structure $\omega$ on $T^*\gamma$.

Now we introduce pensive billiards. Fix an arbitrary smooth {\it delay function} $\ell(p)$ for $p\in (-1,1)$
(alternatively, sometimes it is convenient to rewrite this function as $\tilde \ell(\theta):=\ell(\cos\theta)$).
This function will stand for the  length of the path along the boundary  (in the arc-length parametrization) that a billiard ball will spend  
between hitting the boundary (at the angle $\theta$ where $p=\cos\theta$) and reflecting from it.

\begin{defi}
Given a domain $D\subset \R^2$ with boundary $\gamma = \partial D$ and a function
$\tilde \ell(\theta),$ a {\it pensive billiard} $\mathsf{PB}:\,M \to M$ is a map
which sends $(s_1, \theta_1)\mapsto (s_2+\tilde \ell(\theta_2), \theta_2)$ obtained by traveling the table from $\gamma(s_1)$ in direction making angle $\theta_1$ with $\gamma$ at  that point until one hits $\gamma$ at $\gamma(s_2)$ with the incidence angle $\theta_2$,  then traveling along $\gamma$ distance $\tilde \ell(\theta_2)$ depending on the incidence angle, and then reflecting 
from the point $\gamma(s_2+\tilde \ell(\theta_2))$, according to the standard billiard law, at angle $\theta_2$. 
\end{defi}

\begin{figure}[htb]\centering
    \includegraphics[width=.6\columnwidth]{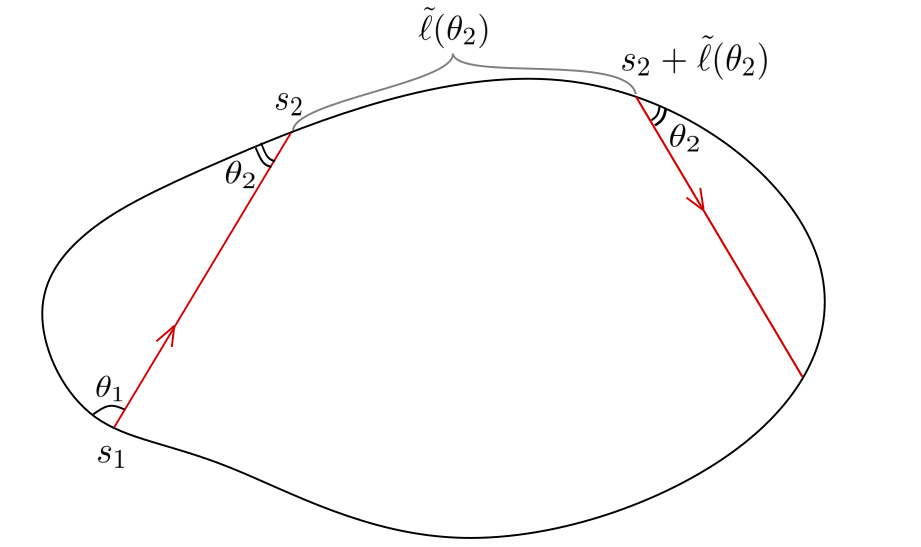} 
  \caption{Pensive Billiard.}
\label{figPensive}
\end{figure}

Alternatively, one can formulate this map in $(s, p)$ coordinates as 
\be 
\mathsf{PB}:(s_1, p_1)\mapsto (s_2+\ell(p_2), p_2)\in T_{< 1}\gamma,
\ee 
resembling the classical billiard which  contemplated  for time $\ell(p)$ before reflecting, hence the name. 

In what follows, we will denote the components of pensive and classical billiard maps by 
\be
(S(s,p), P(s,p)) := \mathsf{PB}(s,p) \quad \text{and} \quad (S_{\mathsf{cl}}(s,p), P_{\mathsf{cl}}(s,p)) := \mathsf{CB}(s,p).
\ee
When it is more convenient to work in the $(s, \theta)$ coordinates, we will write, analogously, 
\be\label{classalp}
(S(s,\theta), \Theta(s,\theta)) := \mathsf{PB}(s,\theta) \quad \text{and} \quad (S_{\mathsf{cl}}(s,\theta), \Theta_{\mathsf{cl}}(s,\theta)) := \mathsf{CB}(s,\theta).
\ee

While  the domain $D$ does not have to be convex for the pensive billiard map to be well defined, there are some additional good properties when this is the case, as we discuss below.  Without convexity, the corresponding billiard maps (both classical and pensive)  might not even be continuous.

\begin{remark}
	Other generalizations of dynamical billiards have been recently studied in 
	\cite{FomenkoVZ, DragovicRadnovich}. In the  planar \textit{billiards with slipping}, billiard ball after hitting the boundary at point $A$ gets transported to the point $\varphi(A)$, where $\varphi$ is a given isometry of the boundary.  An isometry preserving orientation of the boundary curve must be a shift by a constant in the arc-length parameter, and it corresponds to the constant delay function, cf. the example of a satellite billiard below. 
Slipping that changes orientation is related to geodesics on non-orientable surfaces \cite{FomenkoVZ}. 
		In the \textit{magic billiards}, $\varphi$ does not have to be an isometry,  and there are interesting examples related to reflections in two confocal ellipses \cite{DragovicRadnovich}. Our constructions are complimentary, as in the pensive billiards the slipping/delay depends on the angle but not the point of incidence.
\end{remark}

\begin{theorem}
Let $D\subset \mathbb{R}^2$ be  convex.  The 2-form $\omega=dp\wedge dq=\sin \theta\,d\theta \wedge ds$ is $\mathsf{PB}$-invariant (or, equivalently, the pensive billiard map $\mathsf{PB}$ is a symplectomorphism of $M$) for an arbitrary smooth function $\ell$.
\end{theorem}

\begin{proof}
First note that the pensive billiard map $\mathsf{PB}$ is a composition of the classical billiard $\mathsf{CB}$ and a shift map $\mathsf{Sh}: (s, \theta)\mapsto (s+\tilde \ell(\theta), \theta)$. Then the  invariance for $\mathsf{PB}$
immediately follows from that for the classical billiard:
$$
\mathsf{PB}^*\omega=\mathsf{Sh}^*(\mathsf{CB}^*\omega)=\mathsf{Sh}^*\omega=\sin\theta\,d\theta\wedge d(s+\tilde \ell(\theta))=
\sin\theta\,d\theta\wedge ds=\omega\,,
$$
as required.
For completeness, we recall the proof of the corresponding lemma for the classical billiard map (see e.g. \cite{Tab}):

\begin{lemma}
The area form $\omega=\sin \theta \, d\theta\wedge dt$ on $M=T^*_{< 1}\gamma$ is $\mathsf{CB}$-invariant (i.e. the classical billiard map is a symplectomorphism).
\end{lemma}

Indeed, let $f(s_1,s_2)=\|\gamma(s_1)- \gamma(s_2)\|$,
 where $\|\cdot \|$ is the Euclidean distance function in $\R^2$. Then, $\frac{\partial f}{\partial s_2}$ is the projection of $\nabla f$ onto the tangent space to the curve $\gamma$ at $\gamma(s_2)$, so $\frac{\partial f}{\partial s_2}=\cos \theta_2$, since $\nabla f$ is a unit vector making angle $\theta_2$ with (the positive tangent to) $\gamma$. Similarly, $\frac{\partial f}{\partial s_1}=-\cos \theta_1$, so 
	\begin{equation*}
		df=\frac{\partial f}{\partial s_1} ds_1 +\frac{\partial f}{\partial s_2}ds_2=-\cos\theta_1 \, d s_1 +\cos \theta_2 \, d s_2.
	\end{equation*}
	This can be understood as follows:  on $M\times M$ the 1-form in the right-hand side  becomes a complete differential $df$ upon restriction to the graph of the billiard  map $\mathsf{CB}$. Now, by taking one more differential we obtain
	\begin{equation*}
		0=d^2f=d\Big{(} \frac{\partial f}{\partial s_1} ds_1 +\frac{\partial f}{\partial s_2}ds_2\Big{)}=\sin\theta_1 \,d\theta_1 \wedge ds_1 -\sin \theta_2\, d\theta_2 \wedge ds_2\,,
	\end{equation*}
which completes the proof of both the lemma and theorem.
\end{proof}

\subsection{Examples of delay functions}\label{sect:examples} 
Here are three  examples of pensive billiard systems:

  \medskip

    \noindent \textbf{A)} {\it Satellite billiard}  (see also  {\it billiards with slipping} in  \cite{FomenkoVZ}). Imagine that a satellite moving with constant velocity on the orbit receives and sends back a signal according to the billiard law, but requires some time to process the signal.   Assuming that the processing takes the same amount of time independent of the angle of incidence, we arrive at the definition of pensive billiard with constant delay function, $\ell(p)\equiv {\rm const}$.
    The classical billiard corresponds to $\ell(p)\equiv 0$.
     
  \medskip
   
    \noindent \textbf{B)}  {\it Puck billiard} (also called a {\it coin billiard}, see \cite{Bialy}).   
     Consider the geodesic flow on the surface of a cylinder of height $h$ with
base $\gamma=\partial D$, glued on the top and on the bottom to two copies of $D$. On the top and the bottom
of the cylinder the motion is along straight lines, while on the cylindrical part it goes
along geodesics of the cylinder. Identifying top and bottom copies of $D$, such a motion is described by the pensive billiard, where 
the delay function $\tilde \ell(\theta)$ is the base of the right triangle of height $h$ and angle $\theta$ opposite to the base, i.e. $\tilde \ell(\theta)=h\cot\theta$, or equivalently, $\ell(p)=hp/\sqrt{1-p^2}$ for $p=\cos\theta$.

\begin{figure}[htb]\centering
    \includegraphics[width=.5\columnwidth]{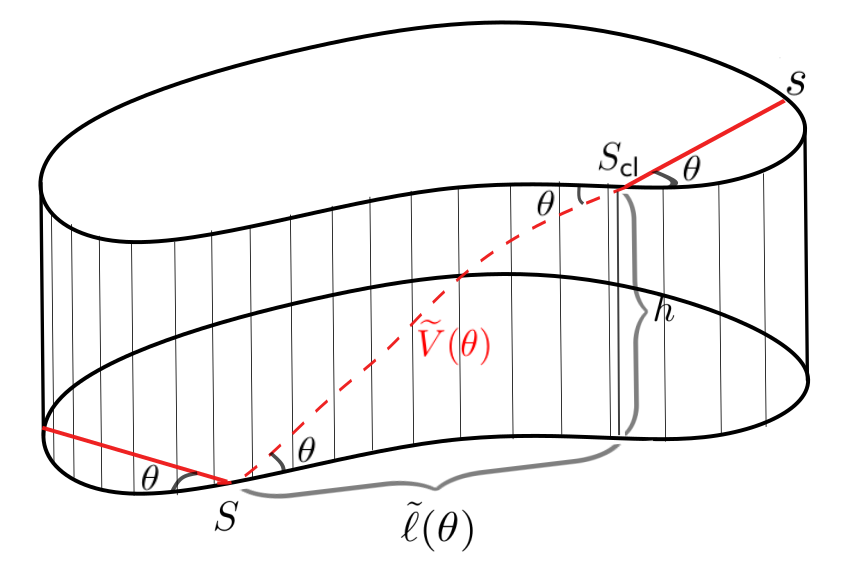} 
  \caption{Puck Billiard. For a puck billiard, generating function \eqref{genfun},  involves potential $V(p)=\widetilde{V}(\theta)$ defined by Equation \eqref{potential} and equal to the length of the geodesic's part  on the side surface of the puck.}
\label{figpuck}
\end{figure}

  \medskip

    \noindent \textbf{C)}  {\it Vortex billiard}. This pensive billiard arises naturally in the motion of point vortex dipoles in a bounded two-dimensional domain $D$. Here we only mention that it corresponds to the delay function
     $\ell(p)=L(1 -p/\sqrt{1+p^2})$, where the length of $\gamma=\partial D$ is equal $2L$.   
     We discuss the vortex billiard in detail in Section \ref{sect:vortex-billiard}.

\begin{figure}[h!]\centering
    \includegraphics[width=.9\columnwidth]{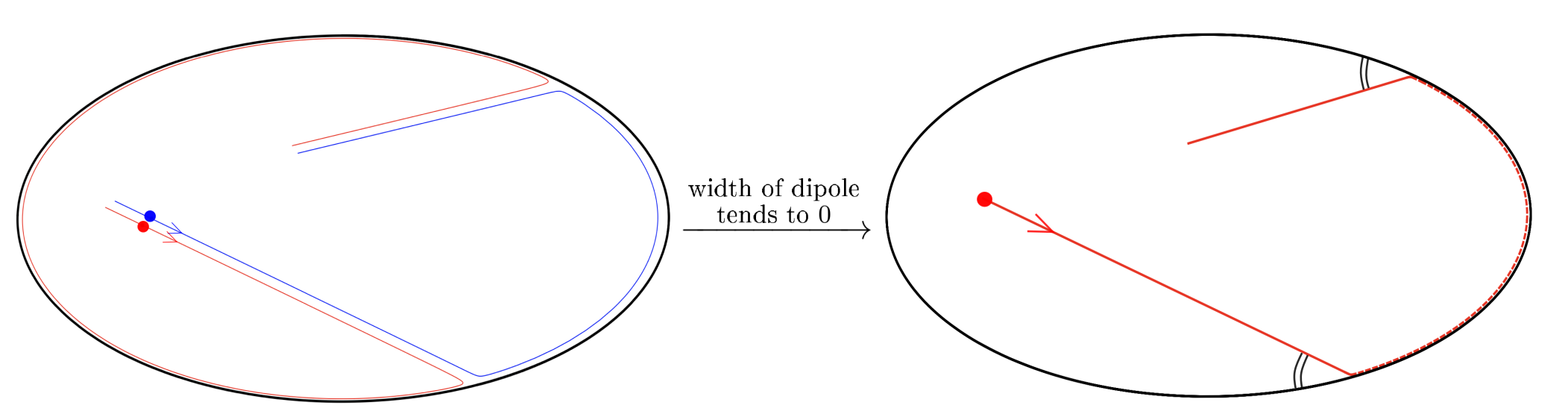} 
  \caption{Pensive Billiard arising as a limit of point vortex motion.}
\label{figVortexBilliard}
\end{figure}

\subsection{Variational principle}
Here we prove that  the pensive billiard with any delay function is subject to a variational principle.  

	Let $\gamma(s)$ be a point on the boundary $\gamma=\partial D$ and $A\in D$ a point inside the domain. Recall that the cosine of the incidence angle at $\gamma(s)$ is
$\cos \theta=p_A(s) =  \frac{(\gamma(s)-A) \cdot \gamma'(s)}{\|\gamma(s) - A\|}.$

	\begin{theorem}\label{thm:potential}
	Given two points  $A, B \in D$, the pensive billiard trajectory that goes from $A$ to $B$ after hitting boundary once must hit $\gamma$ at point $\gamma(s)$, where $s$ is a critical point of the function:
		\be
		f_{A,B}(s) = \|A -\gamma(s)\| + V(p_A(s)) + \|\gamma(s + \ell(p_A(s))) - B\|,
		\ee 
		where a ``potential'' $V$ is  defined by the integral
		\be\label{potential}
		V(p) = \int_0^{p}q\,\ell'({q})\,d{q}\,.
		\ee
	\end{theorem}
	
	\begin{proof}
		For convenience, define 
		$d_A(s) = \|\gamma(s) - A\|$, and observe that 
	$$
	\frac{d}{ds}d_A(s) = \nabla_Q\|Q - A\||_{Q = \gamma(s)}\cdot \gamma'(s)
	= \frac{(\gamma(s)-A)}{\|\gamma(s) - A\|}\cdot \gamma'(s)= p_A(s)\,.
	$$
	We then compute
	\begin{align}
		\frac{d}{ds}f_{A,B}(s) &= \frac{d}{ds}d_A(s) + \frac{d}{ds}V(p_A(s)) + \frac{d}{ds}d_B(s + \ell(p_A(s))) \\
		&= p_A(s) + V'(p_A(s))p_A'(s) + p_B(s + \ell(p_A(s)))(1 + \ell'(p_A(s))p_A'(s))\\
		&= p_A(s) + p_A(s)\ell'(p_A(s))p_A'(s) + p_B(s + \ell(p_A(s)))(1 + \ell'(p_A(s))p_A'(s))\\
		&= (p_A(s) + p_B(s + \ell(p_A(s))))(1 + \ell'(p_A(s))p_A'(s)).
	\end{align}
Finally, recall that pensive billiard trajectories are characterized by $p_A(s) + p_B(s + \ell(p_A(s))) = 0$, equivalent to the equality of the angles of incidence and reflection.
	\end{proof}

\begin{remark}
	The puck billiard is given by $\ell(p) = h \cot(\arccos p) = h\frac{p}{\sqrt{1 - p^2}}$. A direct computation gives that in this case potential \eqref{potential} is
	\be
	V(p) = \int_0^{p}q\,\ell'(q)\,dq = h\int_0^{p}\frac{q}{(1 - q^2)^{3/2}}dq= \frac{h}{\sqrt{1 - p^2}} = \frac{h}{\sin(\arccos p)}.
	\ee	In this case, potential $V$ has a natural interpretation: it is equal to the length of ``curvilinear hypotenuse'', as shown in Figure~\ref{figpuck}. This interpretation also extends to the case of generalized puck billiards 
	(see Section~\ref{sect:puck-billiard}).
\end{remark}

\begin{remark}
Geometrically, the potential 
$$
V(p) = \int_0^{p}q\,\ell'(q)\,dq = p\,\ell(p)- \int_0^{p} \ell(q)\,dq
$$
is equal to the area above the graph of the delay function $\ell(q)$ inside the rectangle $[0,p]\times[0,\ell(p)]$. Such an area naturally arises in the context of generating functions for twist maps, see Section \ref{sect:twist}.
\end{remark}

\subsection{Generating function}
It turns out that a large class of pensive billiards admits explicit generating functions. To specify this class, we introduce the following assumption.
\begin{defi}
	Let $D \subset \R^2$ be convex . We call the pensive billiard in $D$ \textit{locally transitive} at $(s_0,p_0) \in M$ if $\frac{\partial S}{\partial p} (s_0,p_0)\neq 0$   where $(S_0, P_0)=\mathsf{PB}(s_0, p_0)$.  We say that it is \textit{globally transitive} if it is locally transitive everywhere.
\end{defi}

By the implicit function theorem, this condition implies that in a neighborhood $U$ of $(s_0, S_0) \in \gamma\times\gamma$, for any choice of points $(s, S) \in U$, one could smoothly choose a direction $p(s,S)$ such that $\mathsf{PB}(s, p(s,S)) = (S, P)$. (Note that global transitivity is equivalent to the twist condition, as discussed below.)

For classical billiards, this property is a consequence of convexity of the domain. For general pensive billiards, however, this assumption is not satisfied automatically. An example of a situation in which this assumption fails dramatically is a pensive billiard on the unit disk with delay function $\tilde{\ell}(\theta) = -2\theta$. One can easily observe that in this case pensive billiard map is the identity (see Equation \eqref{disk-pb} below), and consequently the billiard is not transitive.

Note that the example above is non-generic even within the class of linear delays: one can see that for any other delay function of the form $\tilde{\ell}(\theta) = C\theta$ with $C \neq -2$, pensive billiard on a disk is transitive. 
\smallskip

Now we can introduce a generating function for  transitive pensive billiards.
\begin{defi}
	Consider a  pensive billiard on $\gamma$ with smooth delay function $\ell$. Assume that it is locally transitive at $(s_0,p_0)$. Given $s, S \in (0, L)$ sufficiently close to $s_0, S_0$ denote by $p^* = p^*(s, S)$ a direction in which the pensive billiard ball need to leave from $\gamma(s)$ to arrive at $\gamma(S)$, i.e. a solution of 
	\be\label{p-star}
	S_{\mathsf{cl}}(s, p^*) + \ell(P_{\mathsf{cl}}(s, p^*)) = S,
	\ee
	the existence of which is guaranteed by transitivity.
	Recall the potential
	\be
	V(p) :=  \int_{0}^{p} q \,\ell'(q) dq,
	\ee
	and define a map
	$H: \mathbb{R}^2 \to \mathbb{R}$ as
	\be\label{genfun}
	\begin{aligned}
	H(s, S) 
	&:= H_{\mathsf{cl}}(s, S_{\mathsf{cl}}(s, p^*)) + V(P_{\mathsf{cl}}(s,p^*)),
	\end{aligned}
	\ee
	where $H_{\mathsf{cl}}(s,S)  := ||\gamma(s) - \gamma(S)||$ is a standard generating function for classical billiards.
	Map $H$ is called a \textit{generating function of a pensive billiard with a delay function} $\ell$.
\end{defi} 
This definition is motivated by the proposition (well known for classical billiards, see \cite{Tab}):
\begin{prop}
	Any locally transitive pensive billiard possesses a generating function 
	$	H(s, S)$  given by \eqref{genfun}.  Namely, if $(S, P) = \mathsf{PB}(s, p)$ then 
	\be\label{pensive-ham}
	\frac{\partial H }{\partial s}(s, S) = -p \quad \text{and} \quad 	\frac{\partial H }{\partial S}(s, S) = P.
	\ee
\end{prop}
\begin{proof}
	Proposition follows from the corresponding one for classical billiards. Recall from \cite{Tab} that for $H_{\mathsf{cl}}$ we have that if $\mathsf{CB}(s, p) = (S_{\mathsf{cl}}, P_{\mathsf{cl}})$, then 
	\be\label{classical-ham}
	\frac{\partial H_{\mathsf{cl}} }{\partial s}(s, S_{\mathsf{cl}}) = -p \quad \text{and} \quad 	\frac{\partial H_{\mathsf{cl}} }{\partial S_{\mathsf{cl}}}(s, S_{\mathsf{cl}}) = P_{\mathsf{cl}}.
	\ee
	In addition, observe that differentiating \eqref{p-star} with respect to $s$ and $S$, we obtain
	\be\label{p-star-diff}
		\begin{aligned}
			\frac{\partial S_{\mathsf{cl}}}{\partial s} + \frac{\partial S_{\mathsf{cl}}}{\partial p}\frac{\partial p^*}{\partial s} + \ell'(P_{\rm cl})\left(\frac{\partial P_{\mathsf{cl}}}{\partial s} + \frac{\partial P_{\mathsf{cl}}}{\partial p}\frac{\partial p^*}{\partial s}\right) &= 0 \quad \text{and}\\
			\frac{\partial S_{\mathsf{cl}}}{\partial p}\frac{\partial p^*}{\partial S} +  \ell'(P_{\rm cl})
			\frac{\partial P_{\mathsf{cl}}}{\partial p}\frac{\partial p^*}{\partial S} &= 1.
		\end{aligned}
		\ee
	Combining \eqref{classical-ham} and \eqref{p-star-diff}, we have
	\be
		\begin{aligned}
			\frac{\partial H}{\partial s}(s, S) &= \frac{\partial  }{\partial s}H_{\mathsf{cl}}(s, S_{\mathsf{cl}}(s, p^*)) + \frac{\partial }{\partial s}\int_{0}^{P_{\mathsf{cl}}(s,p^*)} q \,\ell'(q) dq\\
			&= -p + P_{\mathsf{cl}}\bigg[ \frac{\partial S_{\mathsf{cl}}}{\partial s} + \frac{\partial S_{\mathsf{cl}}}{\partial p}\frac{\partial p^*}{\partial s}\bigg] + P_{\mathsf{cl}} \ell'(P_{\rm cl})\bigg[ \frac{\partial P_{\mathsf{cl}}}{\partial s} + \frac{\partial P_{\mathsf{cl}}}{\partial p}\frac{\partial p^*}{\partial s}\bigg]\\
			&= -p + P_{\mathsf{cl}}\bigg[ \frac{\partial S_{\mathsf{cl}}}{\partial s} + \frac{\partial S_{\mathsf{cl}}}{\partial p}\frac{\partial p^*}{\partial s} +  \ell'(P_{\rm cl})\bigg(\frac{\partial P_{\mathsf{cl}}}{\partial s} + \frac{\partial P_{\mathsf{cl}}}{\partial p}\frac{\partial p^*}{\partial s}\bigg)\bigg]= -p.
		\end{aligned}
		\ee
	Analogously, 
	\be
	\begin{aligned}
		\frac{\partial H}{\partial S}(s, S) &= \frac{\partial  }{\partial S}H_{\mathsf{cl}}(s, S_{\mathsf{cl}}(s, p^*)) + \frac{\partial }{\partial S}\int_{0}^{P_{\mathsf{cl}}(s,p^*)} q \,\ell'(q) dq\\
		&=  P_{\mathsf{cl}}\frac{\partial S_{\mathsf{cl}}}{\partial p}\frac{\partial p^*}{\partial S} + P_{\mathsf{cl}} \ell'(P_{\rm cl})\frac{\partial P_{\mathsf{cl}}}{\partial p}\frac{\partial p^*}{\partial S}= P_{\mathsf{cl}}\bigg[\frac{\partial S_{\mathsf{cl}}}{\partial p}\frac{\partial p^*}{\partial S} +  \ell'(P_{\rm cl})\frac{\partial P_{\mathsf{cl}}}{\partial p}\frac{\partial p^*}{\partial S}\bigg]= P_{\mathsf{cl}}= P.
	\end{aligned}
	\ee
\end{proof}

Proposition above allows for the description of periodic pensive billiard trajectories as critical points of some action. Indeed, for $N$ points $s_1, ..., s_N$ on $\gamma$, suppose that there exist $p_1, ..., p_N$ such that a billiard is locally transitive at $(s_i, p_i)$ and $S(s_i, p_i) = s_{i + 1}$. Then locally around $(s_i, s_{i+1})$ we can define generating functions $H_i$. With that, define (with the convention $s_{N + 1} = s_1$)
\be
H^N(s_1, ..., s_N) := \sum_{i = 1}^N H_i(s_i, s_{i+1})\,.
\ee 
\begin{cor}
A sequence of points $\{\gamma(s_1), \gamma(s_2), \dots, \gamma(s_N)\}$ constitutes a periodic trajectory if and only if 
\be
\frac{\partial H^N}{\partial s_i} = 0
\ee
for $i = 1, ..., N$. In this case, cosines of angles of incidence at points $\gamma(s_i)$ are given by 
\be
p_i = -\frac{\partial H_i}{\partial s_i}(s_i, s_{i +1}).
\ee
\end{cor}


\subsection{Twist property}\label{sect:twist}
The next natural question is under which conditions pensive billiard is a twist map. 
\begin{defi}{(see e.g. \cite[Definition 9.3.1]{Katok_Hasselblatt})}
	A diffeomorphism $f:S^1 \times (0,1)\rightarrow S^1 \times (0,1)$ is called a \textit{right (respectively, left) twist map}, if its lift $F = (F_1, F_2)$ to the universal cover satisfies 
	\be
	\partial_2 F_1 > 0 \quad (\textit{respectively,}\ \partial_2 F_1 < 0).
	\ee 
\end{defi}

	Recall that by definition  for the pensive billiard we have $	S = S_{\mathsf{cl}} + \tilde{\ell}(\Theta_{\mathsf{cl}})$ so that
	\be
	\frac{\partial S}{\partial \theta} = \frac{\partial S_{\mathsf{cl}}}{\partial \theta} + {\tilde{\ell}}'(\Theta_{\mathsf{cl}})\frac{\partial \Theta_{\mathsf{cl}}}{\partial \theta}.
	\ee
	For the classical billiards in smooth convex domains we have (see \cite[Part V, Theorem 4.2]{Katok_1986})
	\be
	\nabla \mathsf{CB}(s, \theta) = 
	\begin{pmatrix}
		\frac{\partial S_{\mathsf{cl}}}{\partial s}      & \frac{\partial S_{\mathsf{cl}}}{\partial \theta}\\[1ex]
		\frac{\partial \Theta_{\mathsf{cl}}}{\partial s}      & \frac{\partial \Theta_{\mathsf{cl}}}{\partial \theta}
	\end{pmatrix}
	=
	\begin{pmatrix}
	\frac{\kappa(s)d(s, S_{\mathsf{cl}}) - \sin\theta}{\sin \Theta_{\mathsf{cl}}}      & \frac{d(s, S_{\mathsf{cl}})}{\sin \Theta_{\mathsf{cl}}}\\[1ex]
	\frac{\kappa(S_{\mathsf{cl}})\kappa(s)d(s, S_{\mathsf{cl}}) - \kappa(S_{\mathsf{cl}})\sin\theta - \kappa(s)\sin \Theta_{\mathsf{cl}}}{\sin \Theta_{\mathsf{cl}}}      & \frac{\kappa(S_{\mathsf{cl}})d(s, S_{\mathsf{cl}}) - \sin \Theta_{\mathsf{cl}}}{\sin \Theta_{\mathsf{cl}}}
	\end{pmatrix},
	\ee
where $d(s, S_{\mathsf{cl}})$ is the Euclidean distance between points $\gamma(s)$ and $\gamma(S_{\mathsf{cl}})$, and $\kappa(s)$ is the curvature of $\gamma$ at $\gamma(s)$. Consequently, the right (left) twist condition for the pensive billiard map becomes
	\be
	0 \overset{(>)}{<} \frac{\partial S}{\partial \theta} = \frac{\partial S_{\mathsf{cl}}}{\partial \theta} + \tilde{\ell}'(\Theta_{\mathsf{cl}})\frac{\partial \Theta_{\mathsf{cl}}}{\partial \theta} = 
	\frac{d(s, S_{\mathsf{cl}})}{\sin \Theta_{\mathsf{cl}}} + \tilde{\ell}'(\Theta_{\mathsf{cl}})\frac{\kappa(S_{\mathsf{cl}})d(s, S_{\mathsf{cl}}) - \sin \Theta_{\mathsf{cl}}}{\sin \Theta_{\mathsf{cl}}},
	\ee
	or equivalently, it can be reformulated as follows.
	\begin{prop} If for all $(s,\theta)\in S^1 \times (0,\pi)$
	\be \label{twist-condition}
	d(s, S_{\mathsf{cl}}) + \tilde{\ell}'(\Theta_{\mathsf{cl}})(\kappa(S_{\mathsf{cl}})d(s, S_{\mathsf{cl}}) - \sin \Theta_{\mathsf{cl}}) \overset{(<)}{>} 0,
	\ee
	then the pensive billiard map is a right (left) twist map.
	\end{prop}
We remark that the twist property is by no means automatic. The following counterexample shows that for any given non-constant delay function $\ell$, the pensive billiard on a sufficiently thin ellipse is not a twist map:
\begin{example}
	Let $D$ be an ellipse with axis lengths $\ve$ and $1/\ve$ where $\ve$ is sufficiently small (see Figure \ref{fignotwist}). Let $\theta_0$ be such that $\tilde{\ell}'(\theta_0) \neq 0$. 
	We claim that $\partial S /\partial \theta$  changes sign and hence the corresponding pensive billiard is not a twist.
\begin{figure}[htb]\centering
	\includegraphics[width=.7\columnwidth]{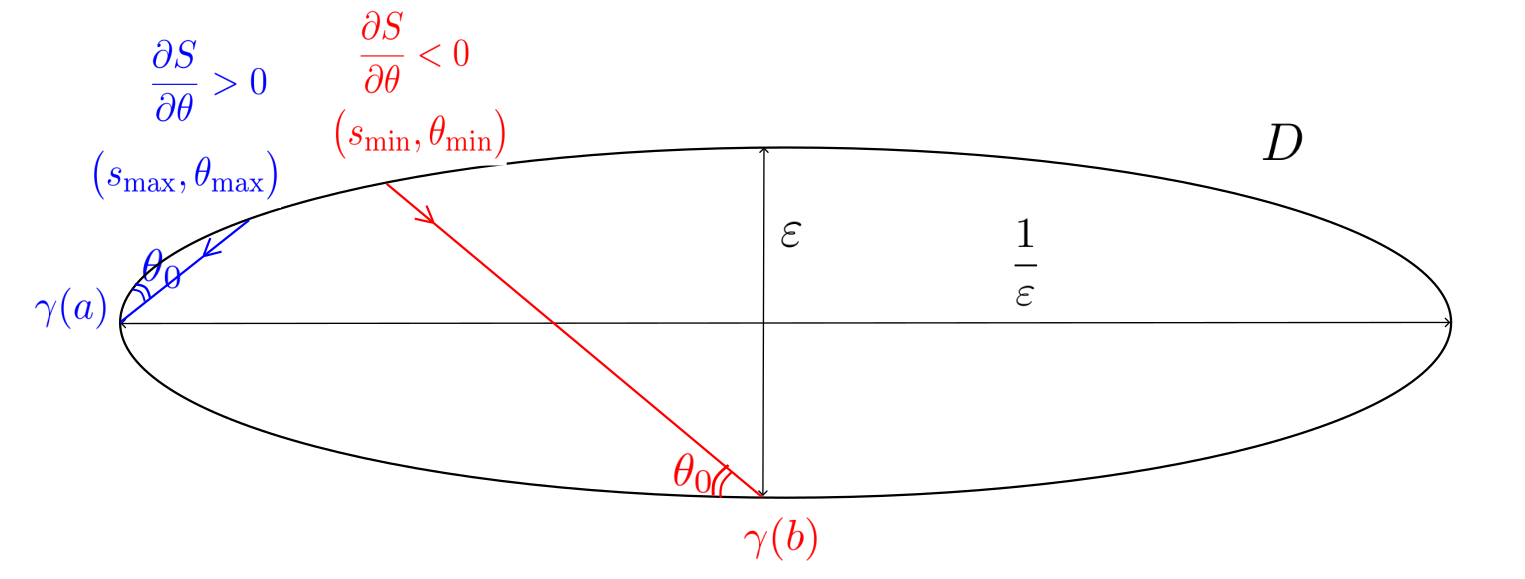} 
	\caption{Pensive billiard map is not a twist map on a sufficiently thin ellipse. In the figure, $\tilde{\ell}'(\theta_0) > 0$ .}
	\label{fignotwist}
\end{figure}

{\rm 	Indeed, compute $\partial S /\partial \theta$ at points $(s_{\rm min/ max}, \theta_{\rm min/ max})$ from which the classical billiard hits at angle $\theta_0$ the ellipse vertices, i.e. the points $\gamma(a)$ and $\gamma(b)$ of minimal and maximal curvature respectively.
	From the proposition above (as well as a  straightforward computation), these derivatives are given by
	\be
	\begin{aligned}
	\frac{\partial S}{\partial \theta}\big(s_{\rm min}, \theta_{\rm min}\big) &= \frac{d(s_{\rm min}, a) + \tilde{\ell}'(\theta_0)(\kappa(a)d(s_{\rm min}, a) - \sin \theta_0)}{\sin\theta_0} \\
	&= \frac{1}{\sin\theta_0}\left(\frac{2\ve\sin\theta_0}{\sin^2\theta_0 + \ve^4\cos^2\theta_0} + \tilde{\ell}'(\theta_0)\bigg(\ve^3\frac{2\ve\sin\theta_0}{\sin^2\theta_0 + \ve^4\cos^2\theta_0} -\sin\theta_0\bigg)\right)\\
	&= -\tilde{\ell}'(\theta_0) + O(\ve),
	\end{aligned}
	\ee
	and
	\be
		\begin{aligned}
		\frac{\partial S}{\partial \theta}\big(s_{\rm max}, \theta_{\rm max}\big) &= \frac{d(s_{\rm max}, b) + \tilde{\ell}'(\theta_0)(\kappa(b)d(s_{\rm max}, b) - \sin \theta_0)}{\sin\theta_0} \\
		&= \frac{1}{\sin\theta_0}\left(\frac{2\ve^3\sin\theta_0}{\cos^2\theta_0 + \ve^4\sin^2\theta_0} + \tilde{\ell}'(\theta_0)\bigg(\frac{1}{\ve^3}\frac{2\ve^3\sin\theta_0}{\cos^2\theta_0 + \ve^4\sin^2\theta_0} -\sin\theta_0\bigg)\right)\\
		&=\tilde{ \ell}'(\theta_0)\bigg(\frac{2}{\cos^2\theta_0} - 1\bigg) + O(\ve).
	\end{aligned}
	\ee
	Consequently, for sufficiently small $\ve$, $\partial S /\partial \theta$ changes sign, so the twist condition is not satisfied. 
	
	This calculation was done with the assumption of a fixed delay function and varying domain, so it does not immediately apply to vortex billiards (for which the delay function is predetermined by the domain's perimeter). However, the same counterexample works in the case of vortex billiard as well, since for any fixed $\theta_0$,  the derivative $\tilde{\ell}'(\theta_0) = L \frac{\sin \theta_0}{(1 + \cos^2\theta_0)^{3/2}} \sim \ve^{-1}$ still dominates $O(\ve)$ term for small $\ve$.
	}
\end{example}
\begin{example}
	For a disk of radius $R$ the condition is simplified to
	\be
	2R\sin \Theta_{\mathsf{cl}} + \tilde{\ell}'(\Theta_{\mathsf{cl}}) \left(\frac{1}{R}\cdot 2R\sin\Theta_{\mathsf{cl}} - \sin\Theta_{\mathsf{cl}}\right) \overset{(<)}{>} 0 \quad\Longleftrightarrow\quad \tilde{\ell}' \overset{(<)}{>} - 2R.
	\ee
	For the puck billiard's delay function $\tilde{\ell}(\theta) = h \cot \theta$, the right twist condition is never satisfied unless $h = 0$, while the left twist condition becomes 
	$
	h > 2R.
$
	Heuristically, this means that for sufficiently tall puck, as the angle of incidence grows, the decrease of the shift along the boundary dominates the increase of the arc-length coordinate due to the classical billiard motion.
\end{example}
These examples motivate the following theorem, which says that a large class of pensive billiards on domains with controlled curvature are twist maps:
\begin{theorem}\label{comparison_thm}
	Let $D$ be a smooth convex domain with the boundary curvature $\kappa \in [1/R, 1/r]$ for some 
	$0<\frac{R}{2}<r<R$. Then any pensive billiard on $D$ with the delay function $\ell$ satisfying 
	\begin{itemize}
		\item$\tilde{\ell}' > -\frac{2r}{2({R}/{r})-1}$ is a right twist map,
		\item  $\tilde{\ell}' < -\frac{2R}{2({r}/{R})-1}$ is a left twist map. 
		\end{itemize}
		In particular, the
		\begin{itemize}
	\item vortex billiard is a right twist map,
	\item sufficiently tall puck  billiard is a left twist map. 
\end{itemize}
\end{theorem}
\begin{proof}
	From an ODE comparison argument that can be found e.g. in \cite[Part 4, \S24]{blaschke1916kreis}, condition $1/R \leq \kappa \leq 1/r$ implies that $\partial D$ is contained in the region between two circles of radii $r$ and $R$ tangent to it at $S_{\mathsf{cl}}$ (see illustration in Figure \ref{comparison_fig}). 
	
\begin{figure}[htb]\centering
	\includegraphics[width=.45\columnwidth]{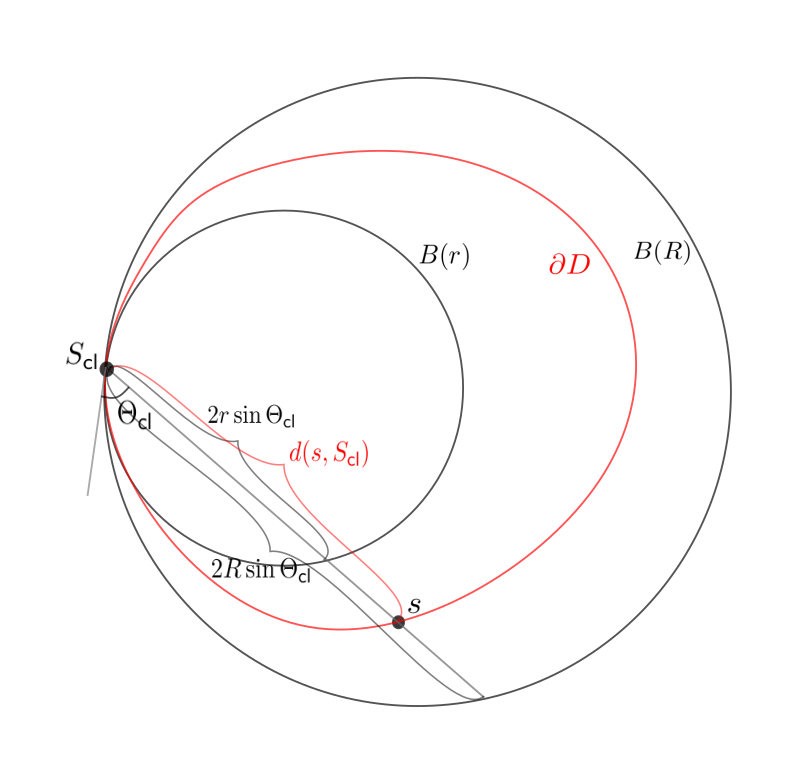} 
	\caption{A curve with curvature bounded between $\frac{1}{R}$ and $\frac{1}{r}$ lies between circles of radii $r$ and $R$ tangent to it.}
	\label{comparison_fig}
\end{figure}

This, in turn, implies that
	\be
	2r\sin\Theta_{\mathsf{cl}} \leq d(s, S_{\mathsf{cl}}) \leq 2R\sin\Theta_{\mathsf{cl}}\quad \Rightarrow \quad  2\frac{r}{R}\sin\Theta_{\mathsf{cl}} \leq \kappa(S_{\mathsf{cl}}) d(s, S_{\mathsf{cl}}) \leq 2\frac{R}{r}\sin\Theta_{\mathsf{cl}}.
	\ee
	Consequently, if $\tilde{\ell}' > -\frac{2r}{2({R}/{r})-1}$
	\be
	d(s, S_{\mathsf{cl}}) + \tilde{\ell}'(\Theta_{\mathsf{cl}})(\kappa(S_{\mathsf{cl}})d(s, S_{\mathsf{cl}}) - \sin \Theta_{\mathsf{cl}}) > 2r\sin\Theta_{\mathsf{cl}} -  \frac{2r}{2\frac{R}{r}-1}\cdot(2\frac{R}{r}-1)\sin\Theta_{\mathsf{cl}} = 0,
	\ee
	so \eqref{twist-condition} implies that $\mathsf{PB}$ is a right twist map. Analogously, if $\tilde{\ell}' < -\frac{2R}{2\frac{r}{R}-1}$
	\be
	d(s, S_{\mathsf{cl}}) + \tilde{\ell}'(\Theta_{\mathsf{cl}})(\kappa(S_{\mathsf{cl}})d(s, S_{\mathsf{cl}}) - \sin \Theta_{\mathsf{cl}}) < 2R\sin\Theta_{\mathsf{cl}} -  \frac{2R}{2\frac{r}{R}-1}\cdot(2\frac{r}{R}-1)\sin\Theta_{\mathsf{cl}} = 0,
	\ee
	so $\mathsf{PB}$ is a left twist map.
	
	Now we verify these conditions on  $\tilde{\ell}'$ for the puck and vortex billiards. 
	
	For the puck billiard with delay function $\tilde{\ell}(\theta) = h\cot\theta$ we have
$
 	\tilde{\ell}'(\theta) = -\frac{h}{\sin^2\theta} \in (-\infty, -h),
$
 	so if $h > \frac{2R}{2({r}/{R})-1}$, the condition of Theorem \ref{comparison_thm} on $\tilde{\ell}'$ is satisfied, and the  puck billiard  is a left twist map.
 	For vortex billiard one has $\tilde{\ell}(\theta) = L\big(1 - \frac{\cos\theta}{\sqrt{1 + \cos^2\theta}}\big)$, which implies that
 	$\tilde{\ell}'(\theta) = L\frac{\sin\theta}{(1 + \cos^2\theta)^{3/2}} > 0 > -\frac{2r}{2({R}/{r})-1}$
 	and, therefore, the vortex billiard map is a right twist map.
\end{proof}

\begin{remark}[Generating function via Twist condition]
  It is known (see \cite[Chapter 9.3]{Katok_Hasselblatt} for an excellent exposition) that for any area-preserving twist map $F: S^1 \times (0,1) \rightarrow F: S^1 \times (0,1)$, one can construct a generating function (i.e. a function satisfying \eqref{pensive-ham}) by setting $H(s, S)$ to be the area bounded between curves $\{S\} \times (0,1)$ and $F(\{S\} \times (0,1))$. In cases when the pensive billiard map is a twist map, this procedure gives the same generating function \eqref{genfun}.  
\end{remark}

Many other implications of the twist property can be found in \cite[Chapter 9.3, 13]{Katok_Hasselblatt}. As a sample of such applications we specialize the Birkhoff theorem on periodic orbits to pensive billiards.

\subsection{Examples of dynamical behaviors} 
We now give three examples where dynamical properties of the pensive billiard system can generally be derived from the classical results.

\subsubsection{Existence of periodic orbits} Recall that a  Birkhoff periodic orbit of type $(p,q)$ is a periodic orbit with period $q$ and rotation number $p/q$, along with some additional properties - see Definition 9.3.6 in  \cite{Katok_Hasselblatt}. We have
\begin{theorem}
	Let $D$ be a smooth convex billiard table with the boundary curvature in between $1/R$ and $1/r$ for some $0 < R/2 < r < R$. Then 
	\begin{itemize}
		\item there exists $h_0 = h_0(r,R)$ such that for any $h > h_0$, and any $p/q \in \mathbb{Q}$ with $p$ and $q$ relatively prime, puck billiard of height $h$ in $D$ has at least two Birkhoff periodic orbits of type $(p,q)$.
		\item for any $p/q \in \big(\frac{1}{2} - \frac{1}{2\sqrt{2}} ,\frac{3}{2} +\frac{1}{2\sqrt{2}} \big) \cap \mathbb{Q}$  with $p$ and $q$ relatively prime, the vortex billiard in $D$ has at least two Birkhoff periodic orbits of type $(p,q)$.
	\end{itemize} 
\end{theorem}
 \begin{proof} In view of Theorem \ref{comparison_thm}, the proof is an application of the Birkhoff theorem: such orbits exists on  twist intervals. Hence one needs to find the corresponding twist intervals for the puck and vortex billiards. For the puck billiard  its twist interval is 
 	\be(\lim_{\theta\rightarrow \pi}S, \lim_{\theta\rightarrow 0}S) = (1 + \frac{1}{2L}\lim_{\theta\rightarrow \pi}\tilde{\ell}(\theta), \frac{1}{2L}\lim_{\theta\rightarrow 0}\tilde{\ell}(\theta)) = (-\infty, \infty),
 	\ee while for the vortex billiard its twist interval is 
 	\be
 	(\lim_{\theta\rightarrow 0}S, \lim_{\theta\rightarrow \pi}S) = (\frac{1}{2L}\lim_{\theta\rightarrow 0}\tilde{\ell}(\theta), 1 + \frac{1}{2L}\lim_{\theta\rightarrow \pi}\tilde{\ell}(\theta)) = ( \frac{1}{2}(1 - \frac{1}{\sqrt{2}}), 1 + \frac{1}{2}(1 + \frac{1}{\sqrt{2}}))\,,
 	\ee
	and the statement follows.
 \end{proof}
\subsubsection{Caustics on the disk}\label{disk_example}
As the simplest example, let us  analyze the case of $D = B_1(0)$. The simplification comes from the fact that the angle  $\theta$ at which the point leaves the boundary is the same as the angle at which it hits next. Consequently, the phase space of any pensive billiard in the disk splits into invariant sets $M_\theta = \gamma \times \{\theta\}$, and the restriction of $\mathsf{PB}$ to each $M_\theta$ is a circle rotation by angle  
\be \label{disk-pb}
\tilde{\theta} = 2\theta + \tilde{\ell}(\theta).
\ee 


Hence any result for a standard billiard in a disk can be adapted to the pensive billiard case.  The induced transformation on this invariant circle is a rotation (depending also on the delay function $\tilde{\ell}(\theta)$)  through a fixed angle. As a consequence, we have, for example:

\begin{prop}\label{caustic}
Any nondegenerate pensive billiard in the disk has concentric circle caustics.
\end{prop}

We will return to this proposition in the context of vortex billiards in Section \ref{sect:numerics}.

\medskip


\subsubsection{Interval exchange maps for polygonal billiard tables}
As with the classical billiards, a rich source of examples is provided by polygonal billiard tables. Some statements translate literally, while some results cease to hold. 
An example of the former is the following. Consider a polygonal billiard table $D$ with angles $\{\frac{2\pi m_i}{n_i}\}$ where $\frac{m_i}{n_i}$ are irreducible fractions with $N = \text{lcm}(n_i)$. Then phase space of the pensive billiard splits into invariant subsets $M_c$, $0 \leq c \leq \pi/ N$, each of which are the products of $D$ and some finite sets (\cite{katok-iet}, \cite{zemlyakov-katok}). Indeed, starting the orbit with any angle $\theta_0$, only a finite number of angles $\theta_k \in \{\theta_0 + \frac{2\pi k}{N}\}$ are accessible, since the pensive billiard map modifies $\theta$ by a linear combination of the angles of $D$. Consequently, the restriction of pensive billiard map to $M_c$ can be realized as a \textit{flow over interval exchange transformation} (see Figure \ref{fig:triangle-iet} for an example), in a complete analogy with the procedure carried in \cite{katok-iet}.
The proof of the following theorem from \cite[Theorem 3]{katok-iet} then translates without modification.

\begin{theorem}
 Let $D$ be a polygon with angles commensurable with $\pi$. The restriction of the pensive billiard flow in $D $ to any manifold $M_c$ is not mixing.
\end{theorem}
\begin{figure}[htb]\centering
	\includegraphics[width=.9\columnwidth]{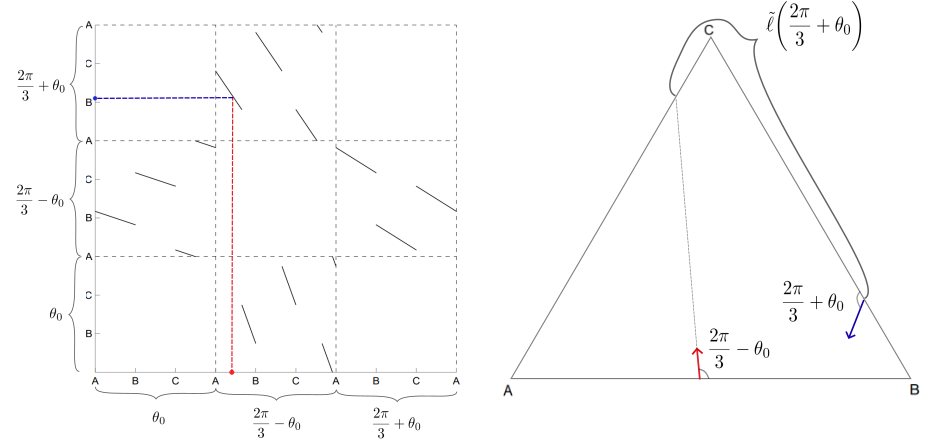} 
	\caption{Left: plot of realization of a pensive billiard on a polygon as an interval exchange transformation. Right: interpretation of the plot. The red vector corresponds to the red point on the plot, the blue vector is its image under $\mathsf{PB}$. Here $\tilde{\ell}(\theta)= \cot\theta$ and $\theta_0=\pi/9$. }
	\label{fig:triangle-iet}
\end{figure}

\newpage

\section{Generalized puck billiards}\label{sect:puck-billiard}
We start by generalizing the definition of a puck billiard by allowing general metrics on the side of the cylinder.
Let $\gamma$ be a smooth convex boundary of a plane domain $D$ as above,
and $N=\gamma   \times [0,h]$ be a cylinder of height $h$ over this curve with coordinates $(s, y)$. 

\begin{defi}
A {\it generalized puck billiard} in $D$ is the geodesic motion on the surface homeomorphic to a sphere formed by gluing the top and bottom copies of $D$ to the side cylinder $N$ with the metric on it  given by $ f(y)\,ds^2 + dy^2$, where $s$ is the arc-length parameter on $\gamma$, while $y\in [0,h]$ is the vertical coordinate on the cylinder, and $f \geq 1$ is a smooth function on $[0,h]$ satisfying  $f(0)=f(h) =1$. 
\end{defi}
 
We note that the requirements  imposed on the metric are natural if one wants to obtain a correspondence between geodesics and a pensive billiard. Indeed, if the metric were dependent on $s$ the corresponding delay function would be dependent on  both the point of incidence and the angle. Moreover, one can see from the proof of Theorem \ref{geodesiconside} that the restriction  $f \ge 1$ is necessary and sufficient for the geodesics starting at the top to reach the bottom.

The puck billiard above corresponds to the flat case
$f(y)\equiv 1$.

\begin{figure}[htb]\centering
    \includegraphics[width=.5\columnwidth]{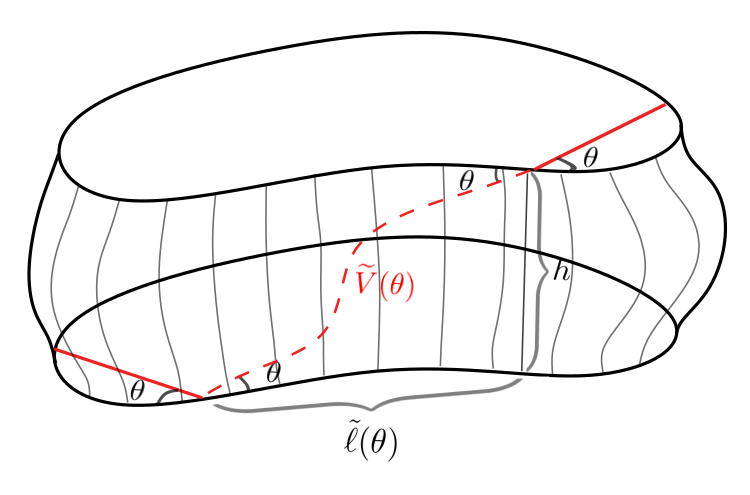} 
  \caption{Generalized Puck Billiard.}
\label{figgenpuck}
\end{figure}

\subsection{Length of geodesics on the side of a cylinder}
It turns out, the potential $V$ defined by Equation \ref{potential} has a natural meaning of the length of ``curvilinear hypotenuse''
in the case of a generalized puck billiard. Indeed, normalize the height of the cylinder: $h=1$.
Assume that a geodesic $\{\lambda(t)\}$ enters the side $N$ with the unit speed
$v_0 = p\partial_s + \sqrt{1 - p^2}\partial_y$, where,   as usual, $p=\cos\theta$, and it exits at the bottom after time $T$.

\begin{theorem}\label{geodesiconside}
For a cylinder with metric $ f(y)\,ds^2 + dy^2$ geodesic segment $\{\lambda(t)\}$ with initial velocity
$v_0 = p\partial_s + \sqrt{1 - p^2}\partial_y$ 
has length 
 	\be
 	\mathcal{V}(p) \big|^T_0 := \int_0^1 \frac{1}{\sqrt{1 - \frac{p^2}{f(y)}}} dy ,
	\ee
and it has horizontal shift (along $s$-coordinate)
	\be \label{Lpform}
	\ell (p)\big|^T_0  = \int_0^1 \frac{p}{f(y)}\frac{1}{\sqrt{1 - \frac{p^2}{f(y)}}}dy.
	\ee
In particular, $\mathcal{V}(p)$ is the potential for the delay function $\ell(p)$: $\mathcal{V}(p) = \int_0^{p}q\,\ell'({q})\,d{q}$.
\end{theorem}

\begin{proof}
	We observe that the corresponding Lagrangian $L(s, y, \dot{s}, \dot{y}, t) = \sqrt{f(y)\dot{s}^2 + \dot{y}^2}$ is invariant under shifts in variables $t$ and $s$. Hence from Noether's theorem $E = f(y)\dot{s}^2 + \dot{y}^2$ and $P = f(y)\dot{s}$ are conserved along the trajectory. Consequently, its length is given by 
	\be
	\mathcal{V}(p)\big|^T_0 = \int_0^T |\dot{\lambda}|dt = \int_0^1 \sqrt{E} \bigg(\frac{dy}{dt}\bigg)^{-1}dy =
	\int_0^1 \frac{\sqrt{E}}{\sqrt{E - \frac{P^2}{f(y)}}} dy.
	\ee
	Similarly, the horizontal shift is given by 
	\be
	\ell(p)\big|^T_0 = \int_0^T \frac{ds}{dt} dt = \int_0^1 \frac{ds}{dt} \bigg(\frac{dy}{dt}\bigg)^{-1} dy = \int_0^1 \frac{P}{f(y)}\frac{1}{\sqrt{E - \frac{P^2}{f(y)}}}dy.
	\ee
	Evaluating $E$ and $P$ at time $t = 0$ we get  $	E = 1$ and $ P  = p$, so that \eqref{Lpform} holds.

	To verify the relation between $\mathcal{V}$ and $\ell$ we compute
		\begin{align}
	\int_0^p q\ell'(q) \,dq &= p \ell(p) - \int_0^p \ell(q)\, dq \\
	&= \int_0^1\left(  \frac{p^2}{f(y)}\frac{1}{\sqrt{1 - \frac{p^2}{f(y)}}} + \sqrt{1 - \frac{p^2}{f(y)}}\right) dy =\int_0^1 \frac{1}{\sqrt{1 - \frac{p^2}{f(y)}}} dy = \mathcal{V}(p),
	\end{align}
	so $\mathcal V$ is the potential for the delay function $\ell$.
\end{proof}
\begin{cor}
	Suppose a pensive billiard with delay function $\ell$ is given. If there exists a function $f: [0,1] \to [1, \infty)$ satisfying the boundary conditions $f(0)= f(1) $ and the relation
$$
	\ell(p) = \int_0^1 \frac{p}{f(y)}\frac{1}{\sqrt{1 - \frac{p^2}{f(y)}}}dy\,
$$
	then  trajectories of the pensive billiard in $D$ naturally correspond  to geodesics on the closed cylinder $(D\times\{0\}) \cup (\partial D \times [0,1]) \cup(D\times\{1\})$  with the flat metric on the top and  bottom and the metric on the  sides given by $f(y)ds^2 + dy^2$.
\end{cor}

\begin{remark}
It would be interesting to classify which pensive billiards allow the puck  interpretation. We note that not all of them do; for instance, as one can see 
from the above theorem, for all generalized puck billiards $\ell(0)=0$. Also, the computation above reveals that for any generalized puck billiard the delay $\ell$ is increasing.
In particular, either of these observations shows that the vortex billiard, whose delay function is $\ell(p):=L(1-\frac{p}{\sqrt{1 + p^2}})$,
is not equivalent to any generalized puck billiard and cannot be described in this way.
\end{remark}


\subsection{From the puck billiards to (almost) classical ones}

Assume that the function $f(y)$ for $y\in [0, h]$ has the symmetry property: 
$f(y)= f(h-y)$. Then instead of defining a delay function, one can consider the classical billiard on the ``half-puck'', simply reflecting geodesic according to the standard billiard law
in the curve $y=h/2$. Such a geodesic will reappear at the top copy of $D$ with the same delay
$\ell(p)$, as in the regular puck billiard.

This suggests the following consideration: extend the flat domain $D$ beyond its boundary 
$\gamma=\partial D$ to a new domain $D_{h/2}$ by a band of width $h/2$ with a special metric $f(y)ds^2 + dy^2$ for 
$y\in [0, h/2]$ and consider the classical billiard in this, generally speaking, non-flat domain  
$D_{h/2}$. Note that geodesics in $D_{h/2}$ refract at the points of $\gamma$, where the flat metric of $D$ changes to the non-flat cylindrical metric.

This point of view can be particularly interesting for small $h$. In particular, it would be interesting to study persistence of caustics or integrability in this enlargement of $D$  for special domains.


\section{Point vortex billiards}\label{sect:vortex-billiard}

In this section, we show that a singular dipole evolving according to 2D Euler equation in the limit of zero separation becomes  a pensive billiard with {\it vortex delay function} 
$$ 
\ell(p):=L-\frac{pL}{\sqrt{1 + p^2}}
$$
for $p=\cos\theta$, where $2L$ is the perimeter of the boundary of the domain $D$.
 The explicit form of the delay function $\ell(p)$ is based on the limiting motion 
of vortex dipoles, studied in Propositions~\ref{prop:fission} and \ref{prop:fusion}.

The goal of this section is to explain why and how this billiard system effectively describes dipoles in the small inter-vortex distance limit.

\subsection{Helmholtz--Kirchhoff point vortex system}
We recall the classical derivation of the point vortex system. On simply connected planar domains (see \cite[\S 2.2]{drivas2023singularity} and \cite{grotta2024interplay} for a discussion of the general case), 
the 2D Euler equation can be written as a closed evolution for the vorticity measure  $\omega=\nabla^\perp \cdot u$:
\begin{align}
\partial_t \omega + u\cdot\nabla\omega = 0, \\
u = \nabla^{\perp} \Delta^{-1}\omega\,,
\end{align}
which manifests in transport of vorticity $\omega$ by a velocity field $u$ generated from $\omega$ through Biot-Savart law. In the point vortex model, we assume that the initial vorticity is given by a finite linear combination of Dirac $\delta$-functions 
$\omega = \sum_{i=1}^{N} \Gamma_i\delta_{\mathbf{z}_i(t)}$
located at points $\mathbf{z}_i = (x_i, y_i)$ in a domain $D \subset \mathbb{R}^2$ and carrying  circulations $\Gamma_i\in \mathbb{R}$.
Formally plugging $\omega$ into the Biot--Savart law to obtain velocity, we find
\be
u(\cdot) = \nabla^{\perp}\Delta^{-1}\omega(\cdot) = \nabla^{\perp}\int_{D}G_D(\cdot ,\mathbf{w})\sum_{i=1}^{N} \Gamma_i\delta_{\mathbf{z}_i(t)}(\mathbf{w}) d\mathbf{w} = \nabla^{\perp}\sum_{i=1}^{N} \Gamma_iG_D({\mathbf{z}_i(t)}, \cdot),
\ee
where $G_D$ is Green's function for the Dirichlet Laplacian on $D$.
Consequently, dynamics of positions of point vortices obey the following system:
\begin{equation}\label{singvorteqns}
	\dot{\mathbf{z}}_i(t) = \nabla^{\perp}\sum_{j=1}^{N} \Gamma_jG_D(\mathbf{z}_i(t),\mathbf{z}_j(t)), \quad   \ i=1, \dots,N.
\end{equation}
Note that this calculation is formal, as the $i = j$ term in the preceding sum is infinite. However, the contribution of self-interaction can be rigorously computed by observing that the singularity is radially symmetric, and the quantity
\be
R_D({\mathbf{z}}) := \lim_{\mathbf{w} \to \mathbf{z}} \Big[ G_D(\mathbf{z}, \mathbf{w}) + \frac{\log(|\mathbf{z}- \mathbf{w}|)}{2 \pi}\Big]
\ee
is bounded (see e.g. \cite{MarchioroPulvirenti}, \cite{DGK24}). Consequently for any radially symmetric desingularization of the vortex, logarithmic singularity does not contribute to the motion, and the equations can be replaced by the regularized version:
\begin{equation}\label{nonsingvorteqns}
	\dot{\mathbf{z}}_i(t) = \nabla^{\perp}\sum_{j\neq i} \Gamma_jG_D(\mathbf{z}_i(t),\mathbf{z}_j(t)) + \nabla^{\perp}\Gamma_iR_D(\mathbf{z}_i), \quad   \ i=1, \dots,N\,.
\end{equation}
Kirchhoff observed that this system can be written in the canonical Hamiltonian form. Namely, the dynamics of $N$ point vortices  is governed by the following Hamiltonian system: 
\begin{equation}
		\Gamma_i \dot{x_i} = \frac{\partial H}{\partial y_i},\qquad 
		\Gamma_i \dot{y_i} = -\frac{\partial H}{\partial x_i}, \qquad \text{for each}  \ i=1, \dots,N\,,
\end{equation}
where the Hamiltonian function $H:D^N\setminus {\bf D} \to \mathbb{R}$ is
$$
H(\mathbf{z}_1,...,\mathbf{z}_N) = \sum_{1\leq i<j\leq N}\Gamma_i\Gamma_j G_D (\mathbf{z}_i, \mathbf{z}_j) + \frac{1}{2}\sum_{i=1}^{N}\Gamma_i^2R_D(\mathbf{z}_i)
$$
(where $ {\bf D}$ consists of all the diagonals $\mathbf{z}_i=\mathbf{z}_j$), 
and the symplectic structure on $D^N \subset \mathbb{R}^{2N}$ is given by $\sum_{i=1}^N \Gamma_i\, dx_i\wedge dy_i$.



We note that the motion of even a small number of point vortices in the presence of boundary is far from being completely analyzed. The complete analysis for one vortex in general domains was carried out by Gustafsson in \cite{gustafsson1979}. Details of motion of several vortices in special domains can be found in \cite{kimura,yang,ModinViviani,KhesinWang} and many other works. In particular, 
in the special case of a half-plane and two equal vortices, $\Gamma_1=\Gamma_2$, an interesting phenomenon of leapfrogging occurs \cite{love1893motion,mavroyiakoumou2020collinear,newton2001n,KhesinWang}.

\subsection{Derivation of vortex billiard system}
The vortex billiard is derived as a limit of zero separation of a vortex dipole evolving as Kirchhoff point vortices.   Specifically, consider two vortices with circulations  $\ve \Gamma$ and $-\ve \Gamma$ for $\Gamma\in \mathbb{R}$, initially located at 
 \be
\mathbf{z}^\ve_\pm := ({x}_0 \pm \ve\cos\theta, {y}_0 \pm \ve \sin \theta ),
 \ee
 for some $(x_0,y_0)\in D$, provided $2\ve< {\rm dist}((x_0,y_0),\partial D)$.
We separate the motion into three regions treated in the following steps.
\vspace{2mm}

\noindent \textbf{Step 1}: (far from boundary)
We claim that away from the boundary singular dipole moves in a straight line, with the relative distance between vortices being constant to first order. To be more precise, let $\delta = \ve^h$ for some $1/2<h<1$, and let 
\be
D^{in}_\delta = \{\mathbf{z} \in D \ |\  \text{dist}(\mathbf{z}, \partial D) > \delta\}.
\ee
We observe that restricted to $ D^{in}_\delta \times  D^{in}_\delta$, Green's and Robin's function satisfy
\be
\begin{aligned}
\nabla G(\mathbf{z}, \mathbf{w})&= -\frac{1}{2 \pi} \frac{\mathbf{z}-\mathbf{w}}{\|\mathbf{z}-\mathbf{w}\|^2} + o\bigg(\frac{1}{\|\mathbf{z}-\mathbf{w}\|}\bigg),\qquad \nabla R(\mathbf{z}) = o\bigg(\frac{1}{\|\mathbf{z}-\mathbf{w}\|}\bigg),
\end{aligned}
\ee
as $\|\mathbf{z}-\mathbf{w}\|\to 0$.
Consequently, to the leading order the equations of motions are given by those in the plane, where  the fact that the dipole moves along a straight line with constant speed is evident. See, for instance, \cite{boatto2015, DGK24} for further details.
\vspace{2mm}

\noindent \textbf{Step 2}: (near the boundary). We claim that as the dipole approaches boundary, the interaction is captured by that in the half-plane, which will be described in \S \ref{half-plane}. Specifically, we have
\begin{prop}
	\label{splitting}
	If dipole traveling at speed $v$ hits the boundary at angle $\theta$, it splits into two vortices traveling along the boundary in the opposite directions with speeds 
	\be v_{\pm} =  v\left(\sqrt{1 + \cos^2\theta} \mp \cos{\theta}\right).
	\ee
\end{prop}
\begin{remark}[Silver ratio] In the limiting case of zero angle, i.e.,  $\cos\theta=1$, normalizing $v=1$, the speeds $v_{\pm}$ of the separated vortices are given by the {\it silver ratio} $\chi = 1+ \sqrt{2}$, as
\be
v_+ = \frac{1}{\chi}, \qquad v_- = \chi.
\ee
Alternatively the silver ratio can be observed in the distances to the boundary after splitting, as by Theorem \ref{gustaf} the speed of a single vortex is inversely proportional to the distance to the boundary.

\end{remark}

	\begin{figure}[h!]\centering
	\includegraphics[width = 0.4\columnwidth]{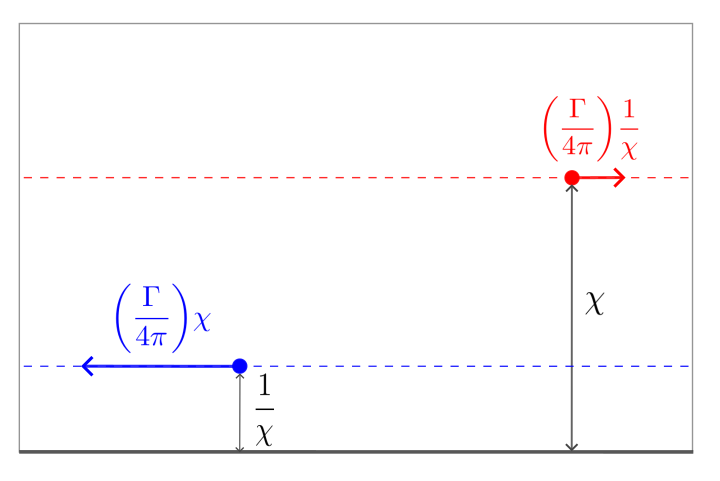}
	\caption{Dipole hitting the boundary ``at zero angle" splits into vortices traveling along the boundary at speeds (respectively, distances to the boundary) proportional (respectively, inversely proportional) to the silver ratio and inverse silver ratio.}
	\label{silver_limit}
\end{figure}

Proof for Proposition \ref{splitting} in the half-plane will be given in the next section, while in the meantime, we provide an argument for the validity of the half-plane approximation.
First, by the discussion above, the vortex dipole Hamiltonian in $ D^{b}_\delta := D \setminus  D^{in}_\delta$ is given by its expression on the whole plane, up to an $o(\ve)$ correction. 
For $x\in M$ lying  within $\delta$-neighborhood of the boundary, the Laplacian can be expressed as
\be
\Delta = \delta^2(\Delta_{\tilde{z},\tilde{s}} + \delta^2 R_\Delta),
\ee
where $\Delta_{\tilde{z},\tilde{s}} = \partial^2_{\tilde{z}} + \partial^2_{\tilde{s}}$ is the Laplacian in the coordinates 
\be
\begin{aligned}
	\tilde{z} &:= \delta\text{ dist}(x, \partial  D)\\
	\tilde{s} &:= \delta\text{ proj}_{\partial D} x,
\end{aligned}
\ee
and $R_\Delta$ is a bounded correction, see  \cite{DIN2023feynman}.
Consequently, Green's functions of $\Delta$ and $\Delta_{\tilde{z}, \tilde{s}}$ in $D^{b}_\delta \cap \{\tilde{s} < \delta\}$ coincide up to $O(\delta^2) = o(\ve)$ correction.  Note that $\Delta_{\tilde{z}, \tilde{s}}$ is essentially the (rescaled) half-plane Laplacian, hence the conclusion.
\vspace{2mm}

\noindent \textbf{Step 3}: (along the boundary)
As vortices leave the region $ D^{b}_\delta \cap \{\tilde{s} < \delta\}$, interaction between them (equivalently, the Green's function term in the Hamiltonian) becomes of lower order compared to the interactions with the boundary (Robin's function term). Consequently, to the leading order the dynamics is given by the following theorem, proved in \cite[Corollary 14]{flucher1997vortex}:
\begin{theorem}[\cite{flucher1997vortex}]\label{gustaf}
	Let $ D\subset\mathbb{R}^2$ be simply connected. A vortex with circulation $\Gamma$ at a small distance $\ve$ from $\partial D$ stays within a distance $\ve + O(\ve^2)$, and travels along the boundary at speed
	\be
	v = -\frac{\Gamma}{4\pi \ve} + O(\ve).
	\ee
\end{theorem}
Since in our renormalization of dipole $\Gamma \sim \ve$,  to the first order the speed of vortices stays as described in Proposition \ref{splitting}. In particular, the boundary curvature does not contribute to the vortex' speed.

\smallskip

Finally, after traveling along the boundary in opposite directions, vortices meet again the distance $2L\frac{v_+}{v_+ + v_-} = L\frac{\sqrt{1 + \cos^2\theta} - \cos\theta}{\sqrt{1 + \cos^2\theta}}$ away from the point of splitting in the direction of motion of the positive vortex. The argument in \textbf{Step 2} can be ran backwards to conclude that vortices will merge into a dipole and enter $ D^{in}_\delta$, and then \textbf{Steps 1-3} repeat. 
\begin{figure}[h!]\centering
        \includegraphics[width=1\columnwidth]{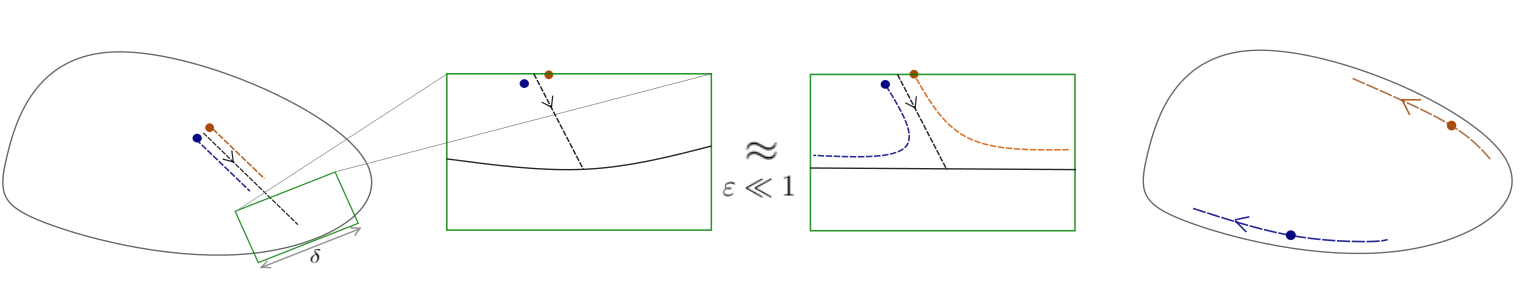}  
  \caption{Cartoon of time evolution of a vortex dipole hitting the boundary.}
\label{fig1}
\end{figure}

This defines the {\it vortex billiard} as the limit of the dipole motion. Next we describe this limit as a pensive billiard with a specific delay function.




\subsection{Dipoles on the half-plane: Fission-fusion rules}
\label{half-plane}

Dipoles in a half-plane can have different types of motion, depending on the strength 
of their interaction as compared to the interaction with their mirrors (or equivalently, ``with the boundary''). 

In the half-plane, Green's function has the following form:
\be
G_{\mathbb{R}^2_+}(\mathbf{z}, \mathbf{w}) = -\frac{1}{2\pi}\log{|\mathbf{z} - \mathbf{w}|} + \frac{1}{2\pi}\log{|\mathbf{z} - \overline{\mathbf{w}}|},
\ee
where $\overline{\mathbf{w}}$ is the mirror image of $\mathbf{w}$. Consequently,
\be
R_{\mathbb{R}^2_+}(\mathbf{z}) := \lim_{\mathbf{w} \to \mathbf{z}} \big[ G_{\mathbb{R}^2_+}(\mathbf{z}, \mathbf{w}) + \frac{\log(|\mathbf{z}- \mathbf{w}|)}{2 \pi}\big] = \lim_{\mathbf{w} \to \mathbf{z}} \big[\frac{1}{2\pi}\log{|\mathbf{z} - \overline{\mathbf{w}}|}\big] = \frac{1}{2\pi}\log{|\mathbf{z} - \overline{\mathbf{z}}|}.
\ee
Combining two preceding equations, Hamiltonian for two vortices in a half-plane is written as 
\be
H(\mathbf{z}_1,\mathbf{z}_2) = \frac{\Gamma_1\Gamma_2}{2\pi} \log{\frac{|\mathbf{z}_1 - \overline{\mathbf{z}}_2|}{|\mathbf{z}_1 - \mathbf{z}_2|}} + \frac{\Gamma_1^2}{4\pi}\log{|\mathbf{z}_1 - \overline{\mathbf{z}}_1|}
+\frac{\Gamma_2^2}{4\pi}\log{|\mathbf{z}_2 - \overline{\mathbf{z}}_2|}.
\ee
For a dipole with $\Gamma_1 = -\Gamma_2 = \Gamma$, we denote for  convenience
\be 
\begin{aligned}
	x_1 &= x_+ \quad x_2 = x_-,\\
	y_1 &= y_+ \quad y_2 = y_-,
\end{aligned}
\ee to simplify the expression for the Hamiltonian. Namely, introducing relative coordinates 
\be
(x_{\mathsf{rel}}, y_{\mathsf{abs}}) := \left(x_+ - x_-, \frac{y_+ + y_-}{2}\right),
\ee
 the Hamiltonian becomes
\be
H(x_{\mathsf{rel}},y_{\mathsf{abs}}) = \frac{-\Gamma^2}{4\pi}\log\bigg[\frac{1}{\mu^2 + x_{\mathsf{rel}}^2} + \frac{1}{4y_{\mathsf{abs}}^2 - \mu^2}\bigg]
\ee
where 
\be\label{mudef}
\mu := y_+ - y_-
\ee
is a conserved ``momentum'' corresponding to the translational invariance. 

Figure \ref{gold_dipole_fig} contains different types of motion for a pair of vortices with the opposite circulations in a nutshell, as the interaction parameter $\mu$ changes.
	\begin{figure}
		\includegraphics[width = 0.9\columnwidth]{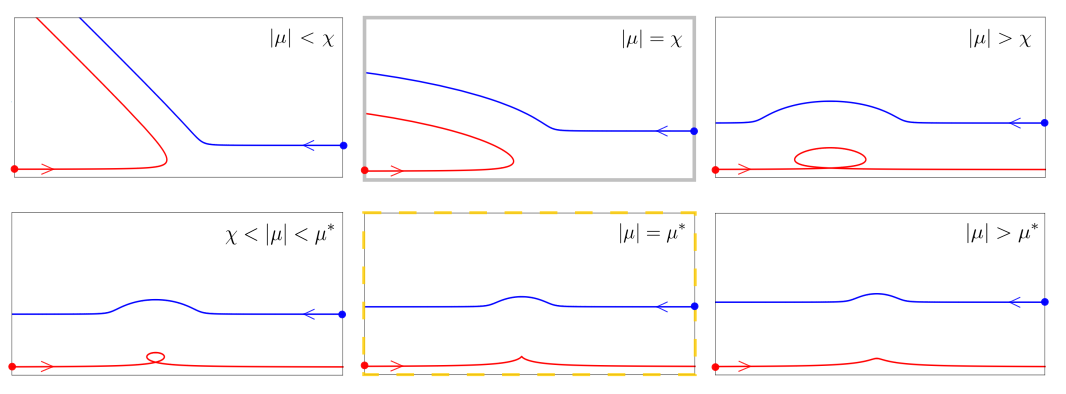}
		\caption{\small Types of motion for two vortices of opposite circulations in the half-plane (bottom of each panel corresponds to the boundary). Top: As the parameter $\mu$ defined in \eqref{mudef} increases through the silver ratio, the dipole stops being formed and vortices start  passing each other. {Bottom: As the parameter $\mu$ increases beyond $\mu^* = 1 + 2\phi + 2\sqrt{1 + 2\phi}\approx 8.35$}, the reverse motion ceases and vortices pass each other smoothly, see \cite{KhesinWang}.}
		\label{gold_dipole_fig}
	\end{figure}
We are interested in the strongest interaction when vortices form a dipole, which is depicted at the upper-left panel.
\medskip

\subsubsection{Fission Rule}
If a vortex dipole `hits' the boundary, it splits  into positive and negative vortices, traveling with different speed along the boundary in the opposite directions.

\medskip

\begin{prop}\label{prop:fission}
 In the half-plane, consider a dipole $\mathbf{z}^\ve_+, \mathbf{z}^\ve_-$ with circulations $\ve\Gamma:=\ve\Gamma_+ = - \ve\Gamma_-$, initially located at $(x_0 \pm \ve\sin\theta, y_0 \mp \ve \cos \theta )$.
 Set  $t_* =4\pi y_0/(\Gamma\sin\theta)$ to be the ``time of hitting the boundary''. 
 Then the motion of the dipole vortices converges to the following trajectories
 $\widetilde{\mathbf{z}}^\ve_+(t), \widetilde{\mathbf{z}}^\ve_-(t)$ in the sense that for  any $t \geq 0$,  $|\mathbf{z}^\ve_\pm(t)- \widetilde{\mathbf{z}}^\ve_\pm(t)|\to 0$
 as $\ve \rightarrow 0$:\\
 \noindent Prior to hitting the boundary (for $t< t_*$):
	\begin{itemize}
		\item $\widetilde{\mathbf{z}}^\ve_i(t) = (x_0 \pm \ve \sin\theta, y_0 \mp \ve\cos\theta) + \frac{\Gamma}{4\pi}t(-\cos\theta, -\sin\theta)$.
	\end{itemize}
	 \noindent After hitting the boundary (for $t>t_*$):
		\begin{itemize}
		 		\item $\widetilde{\mathbf{z}}^\ve_i(t) = \Big((t-t_*)\frac{\pm\Gamma}{4\pi\left(\sqrt{1 + \cos^2\theta} \mp \cos\theta\right)}, \ve\big(\sqrt{1 + \cos^2\theta} \mp \cos\theta\big)\Big)$.
	\end{itemize}
\end{prop}

\begin{proof} From the conservation of Hamiltonian, we can express the equation of the trajectory in $(y_{\mathsf{abs}}, x_{\mathsf{rel}})$--variables:
	\be
	\frac{1}{\ve^2\cos^2\theta + (x_{\mathsf{rel}}/2)^2} + \frac{1}{y_{\mathsf{abs}}^2 - \ve^2\cos^2\theta} = \frac{1}{\ve^2} + \frac{1}{y_0^2 - \ve^2\cos^2{\theta}}.
	\ee
	Let $ 3/4< h< 1$ be fixed. Expressing $y_{\mathsf{abs}}$ as a function of $x_{\mathsf{rel}}$, we see that for any $x_{\mathsf{rel}} > \ve^h$
	\be
	\begin{aligned}
	y_{\mathsf{abs}}  &= \sqrt{\frac{1}{\frac{1}{\ve^2} + \frac{1}{y_0^2 - \ve^2\cos^2{\theta}} - \frac{1}{\ve^2\cos^2\theta + (x_{\mathsf{rel}}/2)^2}} +  \ve^2\cos^2\theta} = \sqrt{\ve^2(1 + \cos^2\theta) + O(\ve^{4-2h})}\\
	& = \ve\sqrt{1 + \cos^2\theta} + O(\ve^{2 - h}),
	\end{aligned}
\ee
	where implicit constant in big-$O$ notation may depend of $\theta, y_{\mathsf{abs}},$ and $h$. Together with the conservation of momentum $\mu = y_+ - y_- = -2\ve\cos\theta$, we get that in $\{x_{\mathsf{rel}} > \ve^h\}$
	\be
	y_\pm = y_\mathsf{abs} \pm \frac{\mu}{2} =  \ve(\sqrt{1 + \cos^2\theta} \mp \cos\theta) + O(\ve^{2 - h}).
	\ee
	Analogously, for any $y_{\mathsf{abs}} > \ve^h$, we have that
$
	x_{\mathsf{rel}} = 2\ve\sin\theta + O(\ve^{2 - h}).
$
	This establishes the shape of the limiting trajectory. 
	
			To obtain the velocity, we also fix  $0< h' < 1-h$ and
	notice that for sufficiently small $\ve$ the whole trajectory lies in the region (see Figure \ref{trajectory-xr-ya})
	\be
	\begin{aligned}
		A_x\cup A_O \cup A_y  :=&  \big\{|x_{\mathsf{rel}} - 2\ve\sin(\theta)|< \ve^{2-h-h'}, y_{\mathsf{abs}}> \ve^h\big\}\cup\big\{0<x_{\mathsf{rel}} < \ve^h,0<y_{\mathsf{abs}} < \ve^h\big\} \cup \\
		&\quad\quad\big\{x_{\mathsf{rel}} > \ve^h,|y_{\mathsf{abs}} -\ve\sqrt{1 + \cos^2\theta}| <\ve^{2-h-h'}\big\}. 
	\end{aligned}
	\ee
	\begin{figure}
		\hspace*{-2cm}     
		\includegraphics[width=0.5\columnwidth]{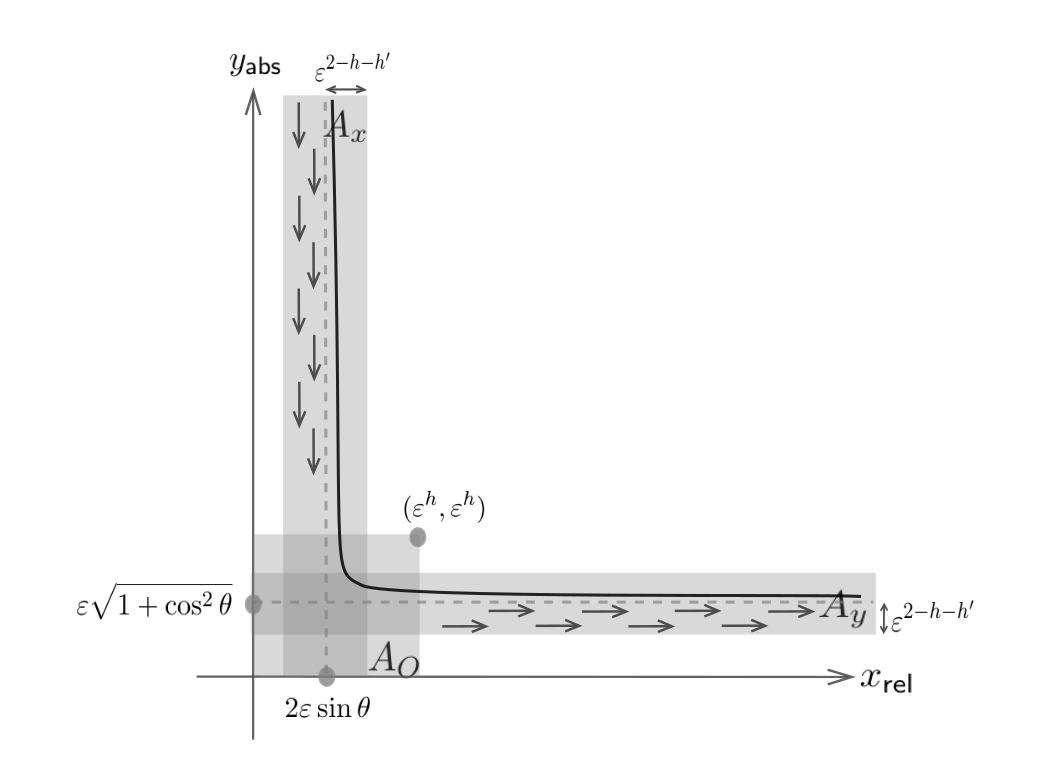}
		\caption{Trajectory of a dipole in $(x_{\mathsf{rel}}, y_{\mathsf{abs}})$ coordinates for small $\ve$.}
		\label{trajectory-xr-ya}
	\end{figure}
	Recalling that evolution of $x_{\mathsf{rel}}, y_{\mathsf{abs}}$ reads
	\be
		\dot{x}_{\mathsf{rel}} =\frac{\Gamma}{4\pi}\frac{\ve}{\frac{1}{4\ve^2} + \frac{1}{y_0^2 - 4\ve^2\cos^2{\theta}}}\frac{8y_{\mathsf{abs}}}{(4y_{\mathsf{abs}}^2 - 4\ve^2\cos^2\theta)^2}, \qquad 
		\dot{y}_{\mathsf{abs}} =\frac{\Gamma}{4\pi}\frac{-\ve}{\frac{1}{4\ve^2} + \frac{1}{y_0^2 - 4\ve^2\cos^2{\theta}}}\frac{2x_{\mathsf{rel}}}{(x_{\mathsf{rel}}^2 + 4\ve^2\cos^2\theta)^2}.
	\ee
	we compute that inside $A_x, A_O$ and $A_y$ we have, respectively, 
	 
	\noindent\begin{minipage}{.3\linewidth}
		\be
		\begin{cases}
			\dot{x}_{\mathsf{rel}} = O(\ve^{3-3h}) \\
			\dot{y}_{\mathsf{abs}} = -\frac{\Gamma}{4\pi}\sin\theta + O(\ve^{1-h-h'})
		\end{cases}\hspace{-2mm},
		\ee
	\end{minipage}\quad
	\begin{minipage}{.3\linewidth}
		\be
		\begin{cases}
			\dot{x}_{\mathsf{rel}} \gtrsim \ve^{3-3h} \\
			\dot{y}_{\mathsf{abs}} < 0
		\end{cases}\hspace{-2mm},{\rm ~and ~}
		\ee
	\end{minipage}
	\begin{minipage}{.3\linewidth}
		\be
		\begin{cases}
			\dot{x}_{\mathsf{rel}} = \frac{\Gamma}{2\pi}\sqrt{1 + \cos^2\theta} + O(\ve^{1-h-h'}) \\
			\dot{y}_{\mathsf{abs}} = O(\ve^{3-3h}) 
		\end{cases}\hspace{-2mm}.
		\ee
	\end{minipage}	

	From the first system we conclude that for $t < t_* - O(\ve^{1-h-h'})$ the trajectory stays in $A_x$ and satisfies
$
	 y_{\mathsf{abs}}(t) = \hat{y} -  \frac{\Gamma}{4\pi}t\sin\theta + O(\ve^{1-h-h'}).
$

From the second system, we conclude that dipole spends at most  time $\ve^{4h-3}$ in $A_O$ before entering $A_y$ at time $t_{**} = t_* + O(\ve^{4h-3})$. 

From the third system we conclude that for $t > t_{**}$ the trajectory stays in $A_y$ and satisfies
$
x_{\mathsf{rel}}(t) = \frac{\Gamma}{2\pi}(t - t_{**})\sqrt{1 + \cos^2\theta} + O(\ve^{1-h-h'}).
$
With that, $x_{\rm abs}(t)$ can be recovered by integrating its evolution equation. Then, individual positions of the vortices can be recovered from absolute and relative coordinates algebraically. Sending $\ve \to 0$ concludes the proof.
\end{proof}	

\subsubsection{Fusion rule}
The following proposition is a counterpart of the previous one, except that we consider merging instead of splitting. First we note that a single vortex of strength $\Gamma$ in the half-plane $y>0$
starting at the point $(x,y)$ moves along the boundary, keeping constant height $y$ at the constant speed $v:= |{\Gamma}/{4\pi {y}}|$ in the direction depending on the sign of $\Gamma$.

\begin{prop}\label{prop:fusion}
	Consider two point vortices $\mathbf{z}^\ve_+, \mathbf{z}^\ve_-$ with circulations $\ve\Gamma:=\ve\Gamma_+ = - \ve\Gamma_-$ in the half-plane, initially located at $({x}_+, \ve\hat{y}_+)$, $({x}_-, \ve\hat{y}_-)$ (with ${x}_+ < {x}_-$). Set $t_* = (x_- - x_+)/(\frac{\Gamma}{4\pi\hat{y}_+} + \frac{\Gamma}{4\pi\hat{y}_-})$ to be the ``time of merging''. As $\ve \rightarrow 0$, the motion converges to the following trajectories $\widetilde{\mathbf{z}}^\ve_+(t), \widetilde{\mathbf{z}}^\ve_-(t)$ in the sense that for any $t \geq 0$,  $|\mathbf{z}^\ve_\pm(t)- \widetilde{\mathbf{z}}^\ve_\pm(t)|\to 0$ as $\ve \rightarrow 0$
	\begin{itemize}
		\item Before merging, for $t<t_*$, the vortices move along the boundary:
		 $$\widetilde{\mathbf{z}}^\ve_\pm(t) = \Big(x_i \pm  t\frac{\Gamma}{4\pi\hat{y}_\pm}, \ve\hat{y}_\pm\Big).$$
		\item If $\chi^{-2} < \frac{\hat{y}_+}{\hat{y}_-} <\chi^2 $ where $\chi$ is the silver ratio, for $t> t_*$ the vortices merge into a dipole:
		
		 $$\widetilde{\mathbf{z}}^\ve_i(t) = (x_* \mp\ve\sqrt{\hat{y}_+\hat{y}_-}  \sin\theta, \pm \ve\sqrt{\hat{y}_+\hat{y}_-}\cos\theta) + \frac{\Gamma}{4\pi\sqrt{\hat{y}_+\hat{y}_-}}t(-\cos\theta, -\sin\theta),$$ 
		 where  $x_* = \frac{x_+\hat{y}_+ + x_-\hat{y}_-}{\hat{y}_+ +\hat{y}_-}$ and $\theta = \tan^{-1}\sqrt{\frac{(\hat{y}_+ - \hat{y}_-)^2}{-\hat{y}_+^2 + 6 \hat{y}_+\hat{y}_- - \hat{y}_-^2}}$.
		\item  Otherwise,  for $t > t_*,$ the vortices pass each other:
				 $$\widetilde{\mathbf{z}}^\ve_\pm(t) = \Big(x_i \pm  t\frac{\Gamma}{4\pi\hat{y}_\pm}, \ve\hat{y}_\pm\Big).$$
	\end{itemize}

\end{prop}
\begin{proof}
	We only prove the convergence at the level of trajectories; upgrading to all-time convergence is analogous to the proof of the `fission' Proposition \ref{prop:fission}. Again from conservation of Hamiltonian,
	\be
	\frac{1}{\mu^2 + x_{\mathsf{rel}}^2} + \frac{1}{4y_{\mathsf{abs}}^2 - \mu^2} = \exp\left({-\frac{4\pi H}{\ve^2\Gamma^2}}\right),
	\ee
	where $\mu := y_+ - y_-$ is the conserved ``momentum'' corresponding to the translational invariance.
	If $\exp\left({-\frac{4\pi H}{\ve^2\Gamma^2}}\right) < {1}/{\mu^2}$, this curve in the $(x_{\mathsf{rel}}, y_{\mathsf{abs}})$-plane has a vertical asymptote, which implies that as vortices approach each other they merge into a dipole. Otherwise, $y_{\mathsf{abs}}$ is bounded, so the vortices pass each other. Evaluating all quantities at the initial data, we observe that 
	\be
	\exp\left({-\frac{4\pi H}{\ve^2\Gamma^2}}\right) = \frac{1}{4\ve^2\hat{y}_+\hat{y}_+} + O(1) \quad\text{and}\quad \frac{1}{\mu^2} = \frac{1}{\ve^2(\hat{y}_+ - \hat{y}_-)^2}.
	\ee
	Consequently, for small $\ve$, condition $\exp\left({-\frac{4\pi H}{\ve^2\Gamma^2}}\right)\leq \frac{1}{\mu^2}$ will be satisfied if and only if
	\be
	(y_+ - y_-)^2 \leq 4y_+ y_- \Leftrightarrow 3  -2\sqrt{2} < \frac{y_+}{y_-} <  3 + 2\sqrt{2}.
	\ee
	Note that, in terms of the silver ratio, $\chi^2=  3 + 2\sqrt{2} $ and $ \chi^{-2} = 3 - 2\sqrt{2} $.
	In this case, as $y_{\mathsf{abs}} \rightarrow \infty$ (corresponding to the limit $t \rightarrow \infty$) we have
$
	\exp\left({\frac{4\pi H}{\ve^2\Gamma^2}}\right) \xrightarrow{t\to \infty} \mu^2 + x_{\mathsf{rel}}^2,
$
	so
	\be
	\tan\theta^\ve(t) = \frac{\mu}{x_{\mathsf{rel}}(t)} \xrightarrow{t\to \infty} \frac{\mu}{\sqrt{\exp\left({\frac{4\pi H}{\ve^2\Gamma^2}}\right) - \mu^2}} \xrightarrow{\ve\to 0} \sqrt{\frac{(\hat{y}_+ - \hat{y}_-)^2}{-\hat{y}_+^2 + 6 \hat{y}_+\hat{y}_- - \hat{y}_-^2}} = \tan\theta.
	\ee
Finally, the width of the merged dipole can again be found from conservation of energy:
\be
\sqrt{\mu^2 + x_{\mathsf{rel}}^2} \xrightarrow{t\to \infty} \exp\left({\frac{2\pi H}{\ve\Gamma}}\right) \xrightarrow{\ve\to 0} \sqrt{4\ve^2\hat{y}_+\hat{y}_-}= 2\ve\sqrt{\hat{y}_+\hat{y}_-}.
\ee
The speed is computed from the fact that it is inversely proportional to the width:
$
v = \frac{\Gamma}{4\pi\sqrt{\hat{y}_+\hat{y}_-}}.
$
\end{proof}

\subsubsection{The golden and silver ratios for vortex dipoles}
The motion of vortex dipoles is somewhat mysteriously related to the golden and silver ratios.
Recall their definitions.

\begin{defi}
The {\it golden ratio} $\phi$ is the ratio of length to width for a rectangle, which preserves this ratio after cutting out {\it the square}. It satisfies the quadratic equation $\phi^2-\phi+1=0$ and explicitly is $\phi=\frac{1+\sqrt 5}{2}= 1.618034...$.

Similarly, the {\it silver ratio} $\chi$ is the ratio of length to width for another rectangle, which preserves that ratio after cutting out {\it two squares}. It satisfies the  equation $\chi^2-2\chi+1=0$ and explicitly is
$\chi=1+\sqrt 2= 2.4142...$.

The golden and silver ratios can be written as continued fractions, $\phi=[1;1,1,1,...]$ and $\chi=[2;2,2,2,...]$. Similarly one can define the bronze ratio, etc, as quadratic irrationals $[n;n,n,n,...]$, which are all called {\it metallic means}.
\end{defi}

Return to point vortices and 
consider the following normalization of the ``dipole at infinity''  in half-plane $\{y>0\}$ with circulations $\Gamma_1 = -\Gamma_2=1$, similar to that used in the fusion above. Recall that 
two point vortices $z_i=( x_i,  y_i)$ have limits $\lim_{t\to-\infty} {y}_i={y}_i^*$. We normalize them  so that $0<y_1^*\le y_2^*$ and for the lowest of them the distance between its  limit ${y}_1^*$ and its mirror image $\bar y_1^*$ equals 1: 
$|{y}_1^*-\bar y_1^*|=1$ (i.e. ${y}_1^*=1/2$ and $\bar y_1^*=-1/2$).

\begin{theorem}
The type of motion for the vortices approaching each other along the boundary changes as follows, depending on the momentum $|\mu|=y_2^*- y_1^*$, see Figure \ref{gold_dipole_fig}:
\begin{itemize}
		\item for $|\mu|$ less than the {\rm silver ratio}, 		
		$0\le|\mu|<\chi$, the vortices  merge into a dipole and travel together (along a slanted or vertical asymptote);
		\item for $|\mu|\ge \chi$, vortices pass one another  and remain at a finite distance from the boundary.
\end{itemize}
\end{theorem}
\begin{proof}
This dichotomy can be regarded as a reformulation of  Proposition \ref{prop:fusion} above, but in the normalization defined by the momentum $\mu$.
Note that if ${y}^*_2/{y}^*_1 = 3 + 2\sqrt{2} $ and $y_1^*=1/2$ (the latter corresponds to $|{y}_1^*-\bar y_1^*|=1$), then $|\mu|={y}^*_2-{y}^*_1 = (3 + 2\sqrt{2})/2-1/2=1+ \sqrt{2}=\chi$.
\end{proof}
\begin{remark}
The {\it golden ratio} also manifests itself in the motion of vortex dipoles. Namely, typically 
vortices pass each other without stopping and their trajectories are smooth submersed curves in the half-plane.
However, there is a special value of the momentum $\mu$, when the lower vortex has an instantaneous stop, and its trajectory has a cusp. Normalize the dipole motion at the {\it cusp moment} rather than at infinity by setting
$|{y}_1-\bar y_1|=1$ (or, equivalently, ${y}_1=1/2$ and $\bar y_1=-1/2$)  at the moment of cusp.
Then at the same instant, the momentum  is $|\mu|={y}_2-{y}_1 = \phi$, where $\phi $ is the golden ratio, see \cite{KhesinWang} and Figure \ref{gold_dipole_fig} above.
\end{remark}


\subsubsection{Vortex pairs, leapfrogging and the metallic means}
It turns out that both  the golden and silver ratios also manifest themselves in the 
motion of vortex pairs in the half-plane, i.e. point vortices of the same sign and strength.
We present it here to demonstrate 
ubiquitous appearance of the metallic means.  
Figure \ref{gold_vortex_pair_fig} contains different types of  motions for a vortex pair for 
different values of $\mu$.

\begin{figure}
\includegraphics[width = 0.9\columnwidth]{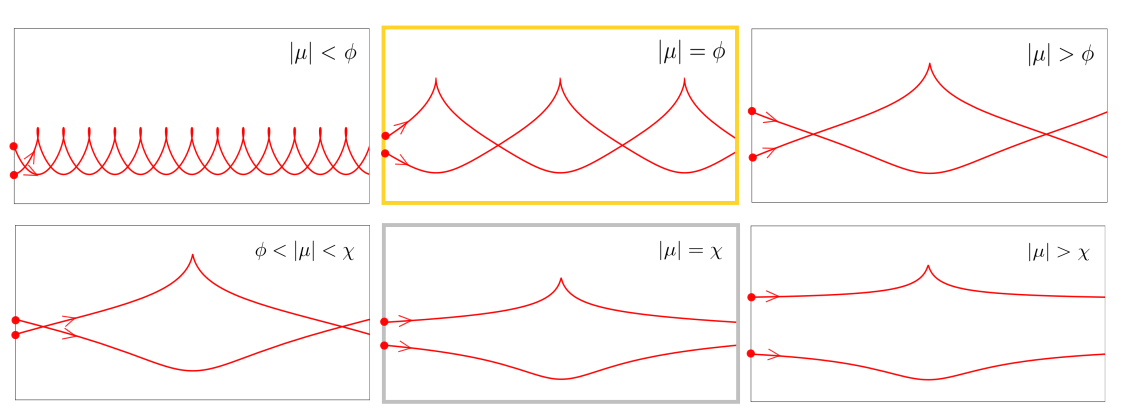}
\caption{\small Types of motion of two vortices of equal circulations in the half-plane. Top: As the parameter $\mu$ defined in \eqref{mudef}  exceeds the golden ratio, the vortices stop reversing. Bottom: As the parameter $\mu$ exceeds the silver ratio, periodic leapfrogging motion ceases and the vortices start passing each other only once. See \cite{KhesinWang}.}
\label{gold_vortex_pair_fig}
\end{figure}
Since vortex pairs behave differently at infinity, they need a different normalization.
We use the fact that they necessarily happen to be at the same vertical at least once. 
Let $0<y_1^*\le y_2^*$ be their $y$-coordinates when the vortices are on the same vertical.
We again normalize them  so that  the distance between the lowest vortex ${y}_1^*$ and its mirror image $\bar y_1^*$ equals 1:  $|{y}_1^*-\bar y_1^*|=1$. Then

\begin{theorem}
The motion for the vortex pair changes as follows depending on $|\mu|:=y_2^*- y_1^*$:
\begin{itemize}
		\item for $0\le|\mu|<\phi$ the vortices  leapfrog with periodic reversing of the top one;
		\item for $|\mu|=\phi$ the upper vortex has an instantaneous stop, and its trajectory has a cusp;
		\item for $\phi\le |\mu|<\chi$ the vortices  have a periodic leapfrogging motion without reversing;
		\item for $\chi\le |\mu|$ the vortices pass each other only once.
\end{itemize}
\end{theorem}

\begin{proof}
The golden ratio part was proved in \cite{KhesinWang}  in the normalization using the cross-ratio
involving 4 points, $y_2^*, y_1^*, \bar y_1^*, $ and $\bar y_2^*$. It is equivalent to the one above upon fixing  $|{y}_1^*-\bar y_1^*|=1$. 

The silver ratio part follows from the paper \cite{love1893motion}. Indeed, 
Love proved that a periodic leapfrogging motion with or without reversing exists under the condition
$3 - 2\sqrt{2} < {y}^*_2/{y}^*_1 < 3 + 2\sqrt{2} $. As we have shown above, it is equivalent to 
$|\mu|<\chi$. 
\end{proof}

There arises a natural question of whether other metallic means also appear in the problem of point vortices or in 2D hydrodynamics in general. 


\subsection{Numerical experiments with vortex dipoles}
\label{sect:numerics}

In this section, we present a few numerical examples of point vortex billiards.  One set of examples is for a family of domains where convexity is eventually lost.  Another example comprises multiple dipole billiards on the same disk, resulting in a chaotic motion.
\vspace{2mm}


Let us start with the case of a single dipole on disk.  Recall that  the motion of two vortices in the disk is integrable
(due to the rotational and time-translation symmetries of the system, the angular momentum and Hamiltonian serve as two integrals of motion in involution.),
see e.g. \cite{geshev2018motion, yang}. 
Both the motion of a non-singular vortex dipole pair on the disk
and the vortex billiard as its limit have caustics, see Figure \ref{fig:caustic} for a demonstration of Proposition \ref{caustic}  for such caustics. In particular, the corresponding trajectories can be found explicitly, rather than numerically. 
Bifurcation diagrams as interaction parameters vary have been studied in \cite{ryabov2019bifurcation, ryabov2021dynamics,sokolov2017bifurcation}.
\begin{figure}[htb]\centering
    \includegraphics[width=.32\columnwidth]{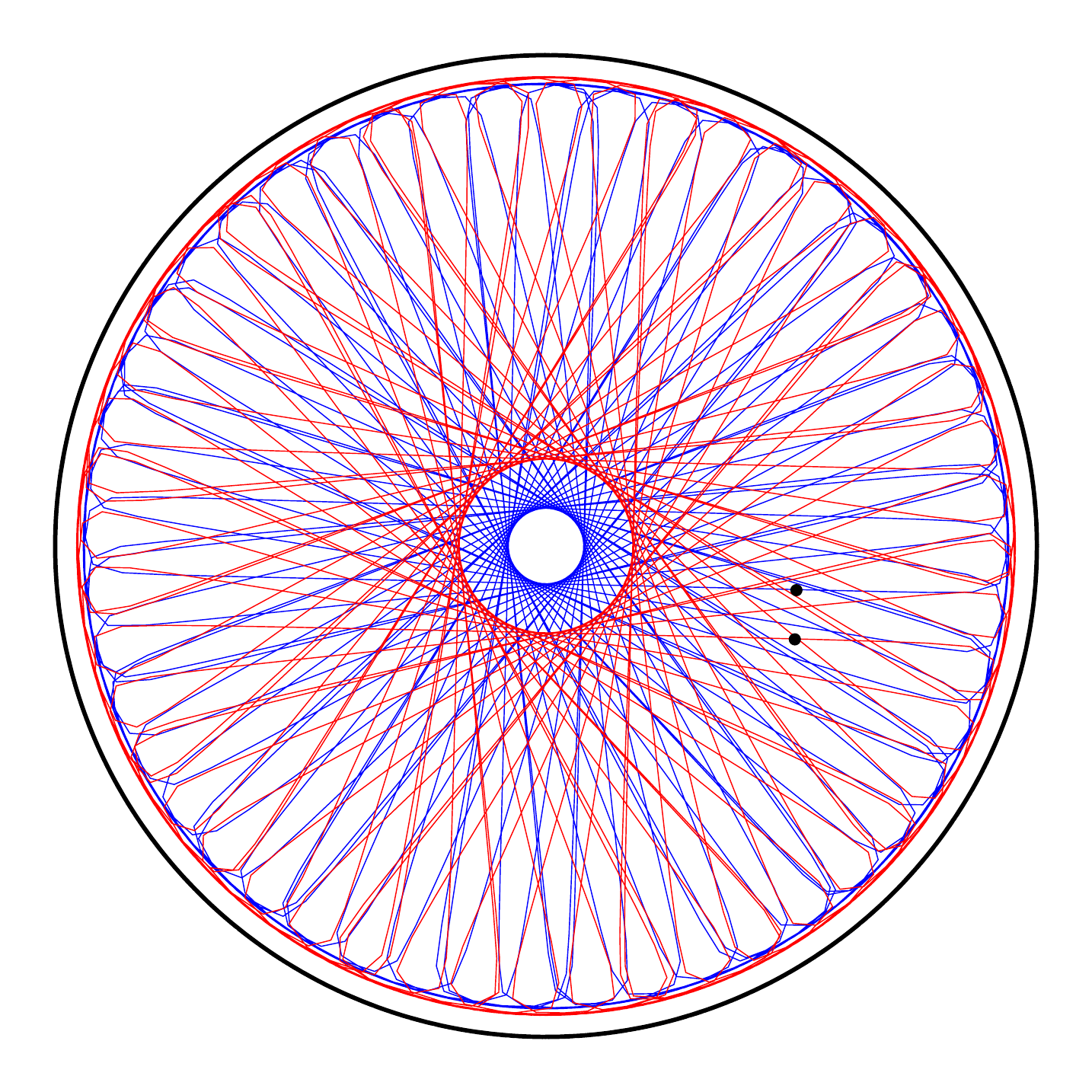} 
        \includegraphics[width=.32\columnwidth]{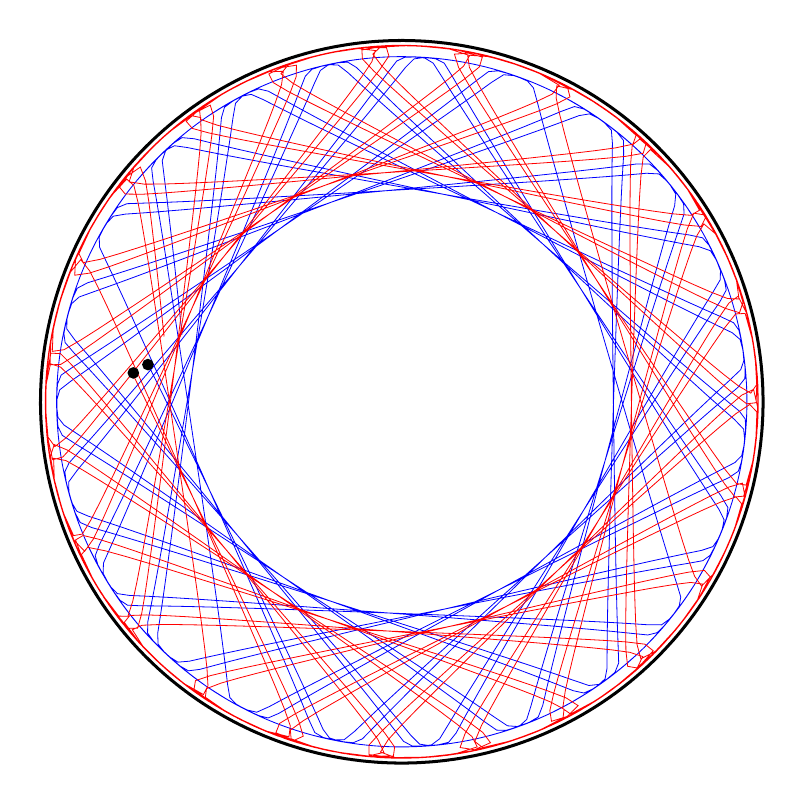} 
                \includegraphics[width=.32\columnwidth]{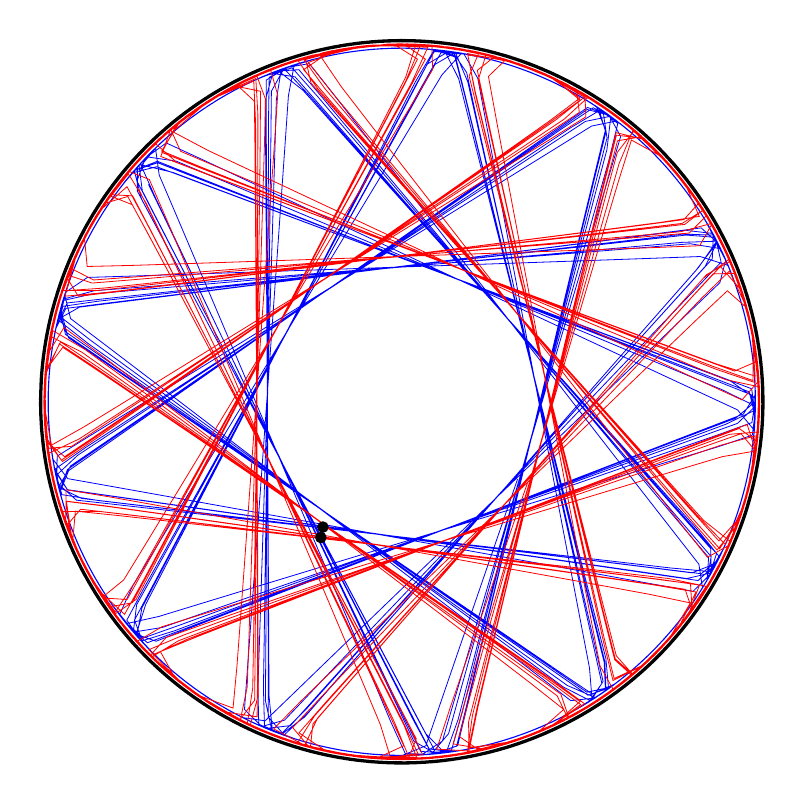} 
  \caption{The caustic circle for one dipole on the disk. Three different initial conditions, with decreasing separations left to right. Trajectories of vortices with positive and negative circulations are shown, respectively, in red and blue.}
\label{fig:caustic}
\end{figure}

Even far from the singular limit, the motion resembles a billiard system (Figure \ref{figdisk}, left panel). 
 On the disk, interaction with image vortices sitting at `inverted' points outside the circle completely characterize the interaction with the boundary (Figure \ref{figdisk}, right panel).

\begin{figure}[htb]\centering
   \includegraphics[width=.36\columnwidth]{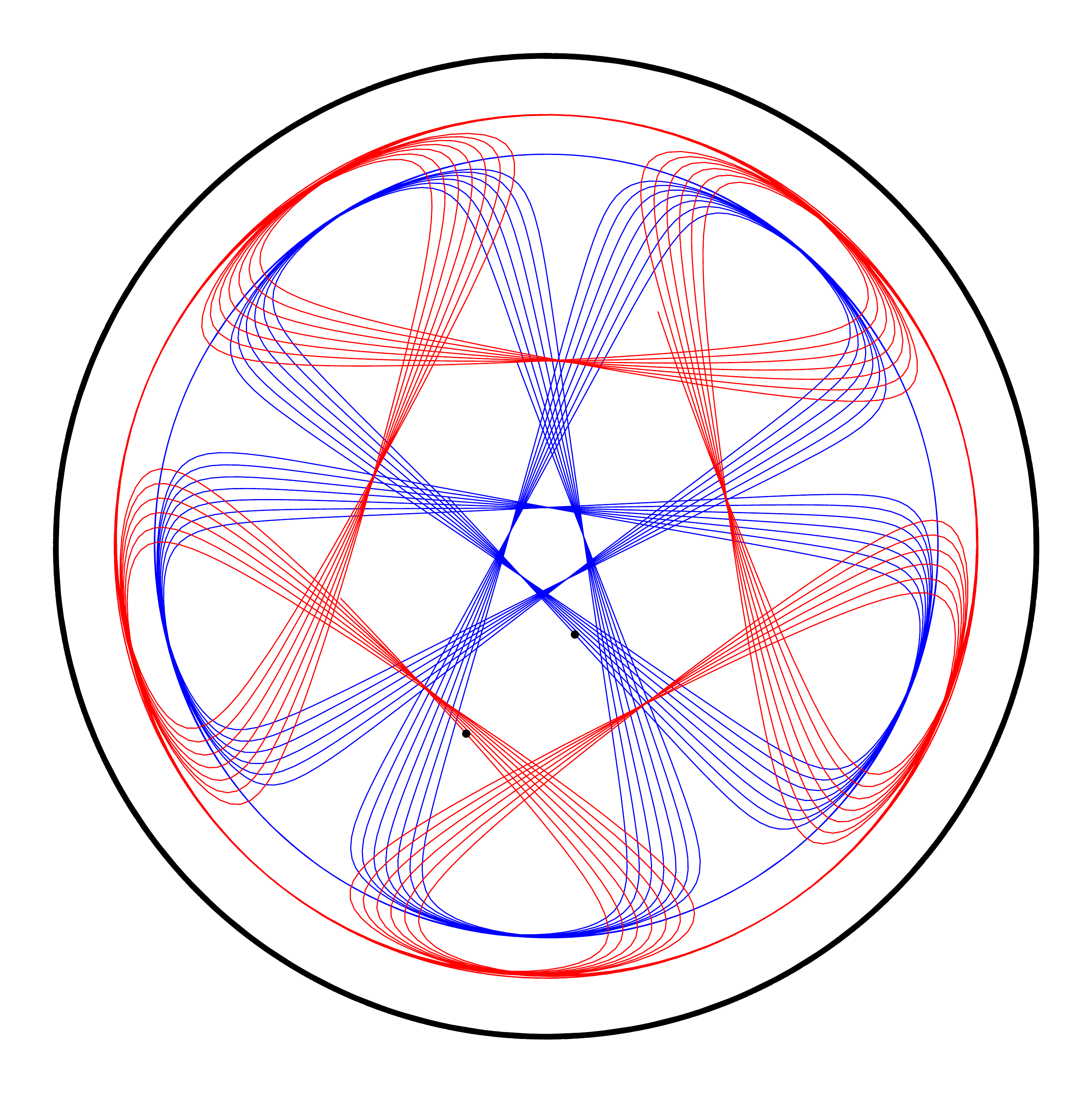} \hspace{6mm}
      \includegraphics[width=.39\columnwidth]{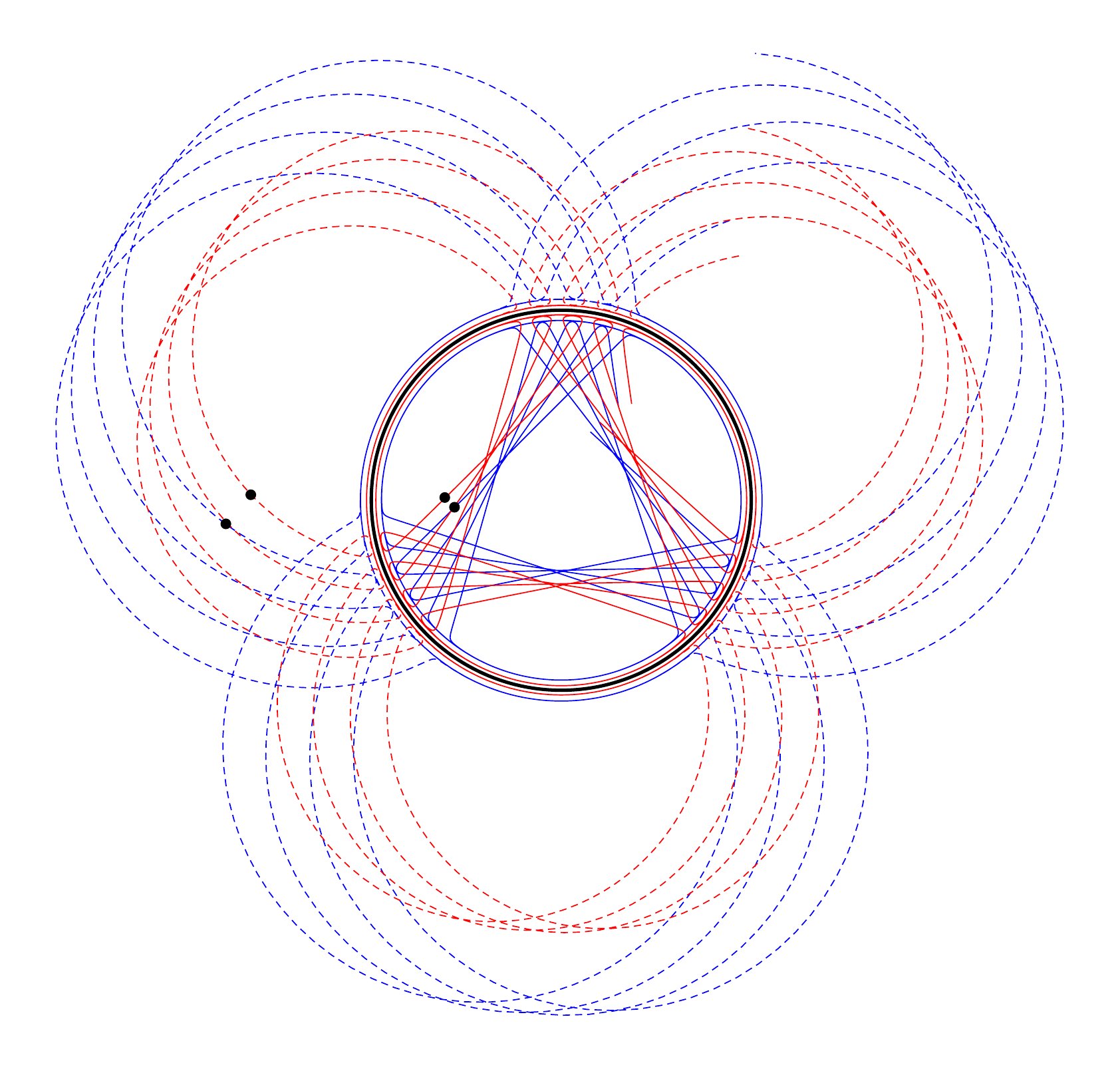} 
  \caption{Left: Vortex dipole on disk with large inter-vortex distance. Right: trajectories of the dipole together with its image charges.}
\label{figdisk}
\end{figure}

\subsubsection{Single dipole on Neumann ovals}
We present here a numerical study of the point vortex system on Neumann ovals, a family of domains which interpolates from a disk domain to a two-disk domain through hourglass-type domains.  All of these can be defined as conformal images of unit disk under the two-parameter ($a,\lambda$) mapping
\be
z=F_{a,\lambda}(Z) := \frac{aZ}{1 - \lambda^2 Z^2}. 
\ee
The parameter $a$ can be chosen to fix the area of the domain to be $\pi$ by taking $a(\lambda) = \frac{1-\lambda^4}{(1+\lambda^4)^{1/2}}$.  
The parameter $\lambda$ controls the shape of the oval: $\lambda=0$ is the disk, while for $\lambda\in (0,1)$ is an interpolating family of hourglass domains which limit as $\lambda\to 1$ to two disks.  

The practical utility of this family of domains is that one can relate  explicit Green's function on the disk domain (obtained by the method of inversion), to  Green's function on the Neumann oval domain via the (inverse) conformal map. Specifically, identifying $\mathbb{R}^2$ with $\mathbb{C}$ and letting $f$ be the inverse map of $F_\lambda$,   Green's function on the Neumann oval is
\be
G(z,w) = G_{\rm disk}(f(z),f(w)), \qquad z,w\in {\rm Neumann\ oval}.
\ee
The Hamiltonian for the vortex system is based on this adjusted Green's function.
We exploit this structure to numerically integrate two point vortices on Neumann oval domains, see Figure \ref{figOval}.
 For more details and interesting simulations of point vortices on these domains, see \cite{ashbee2013generalized} and \cite[Appendix A.2.2]{ashbee2014dynamics}.
 
 We remark that, although the domains are conformally equivalent to the disk, the two-vortex system is not necessarily integrable on the Neumann oval despite being so on the disk.  The reason is that, while the Green's function is conformally invariant, the Robin function is not and hence the corresponding Hamiltonian is not conformally invariant.   It would be interesting to study ergodic properties of pensive billiards on this family of domains, and whether or not caustics survive for small values of $\lambda$.

\begin{figure}[htb]\centering
   \hspace{-5mm} \includegraphics[width=.3\columnwidth]{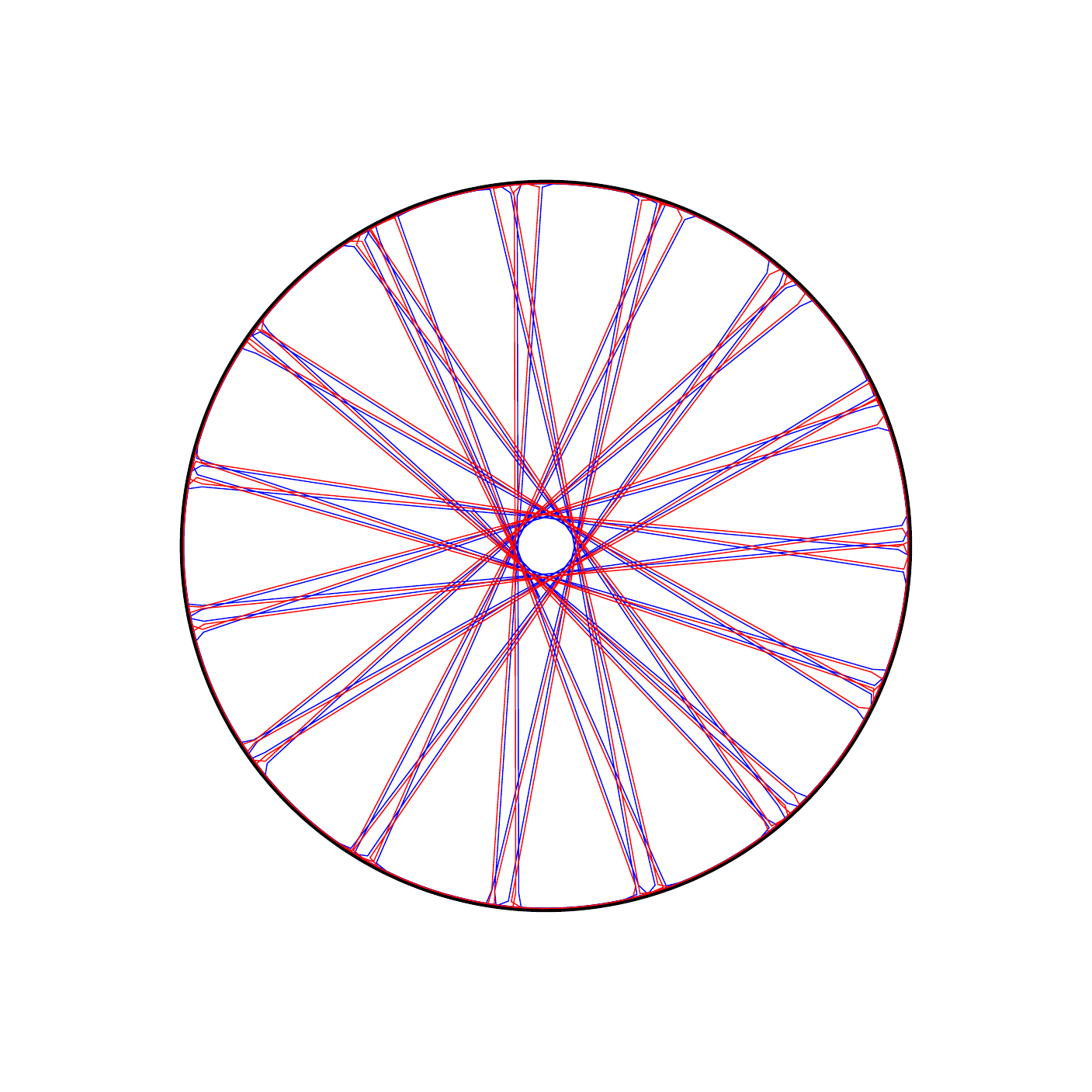} \hspace{-7mm}  \includegraphics[width=.3\columnwidth]{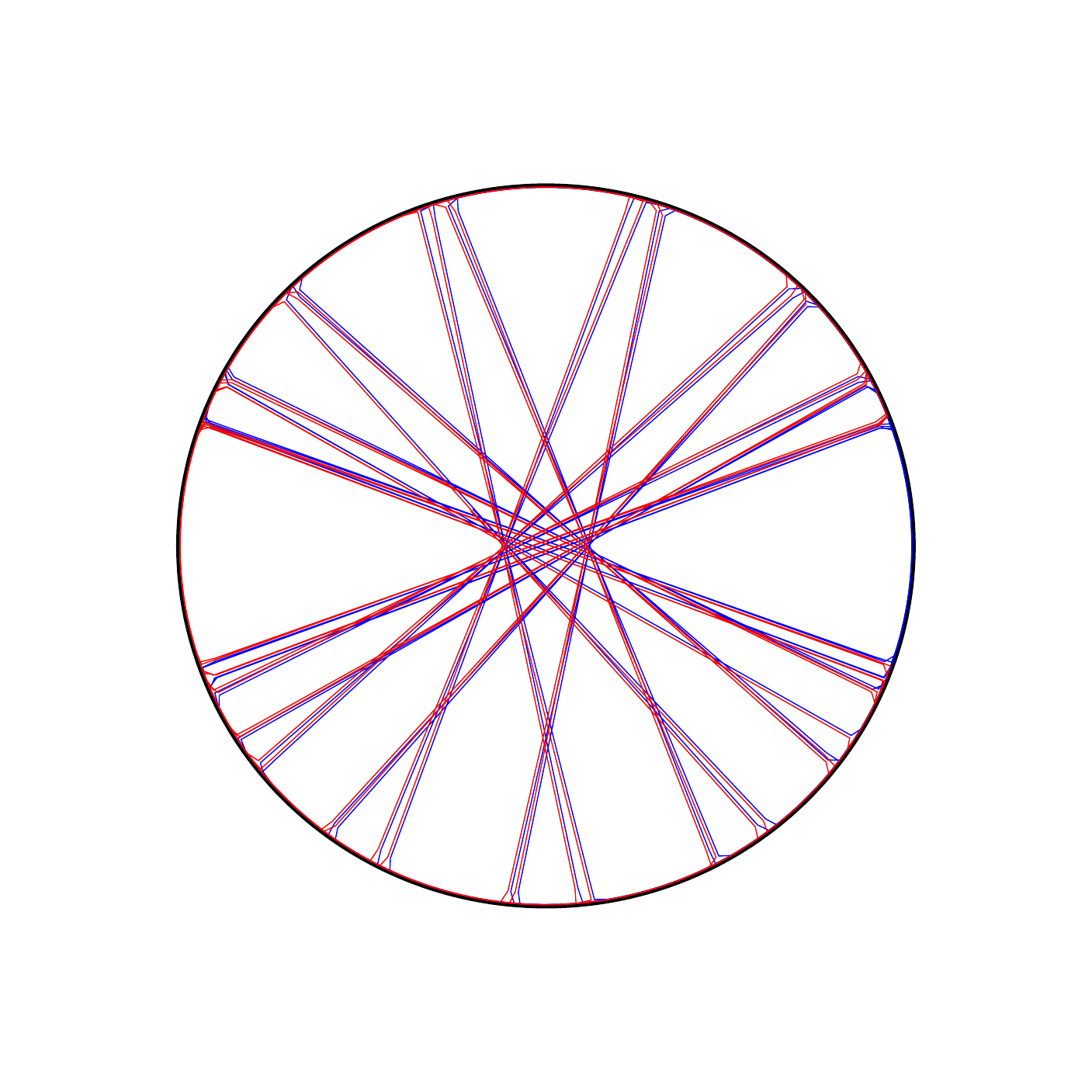} \hspace{-5mm}   
   \includegraphics[width=.3\columnwidth]{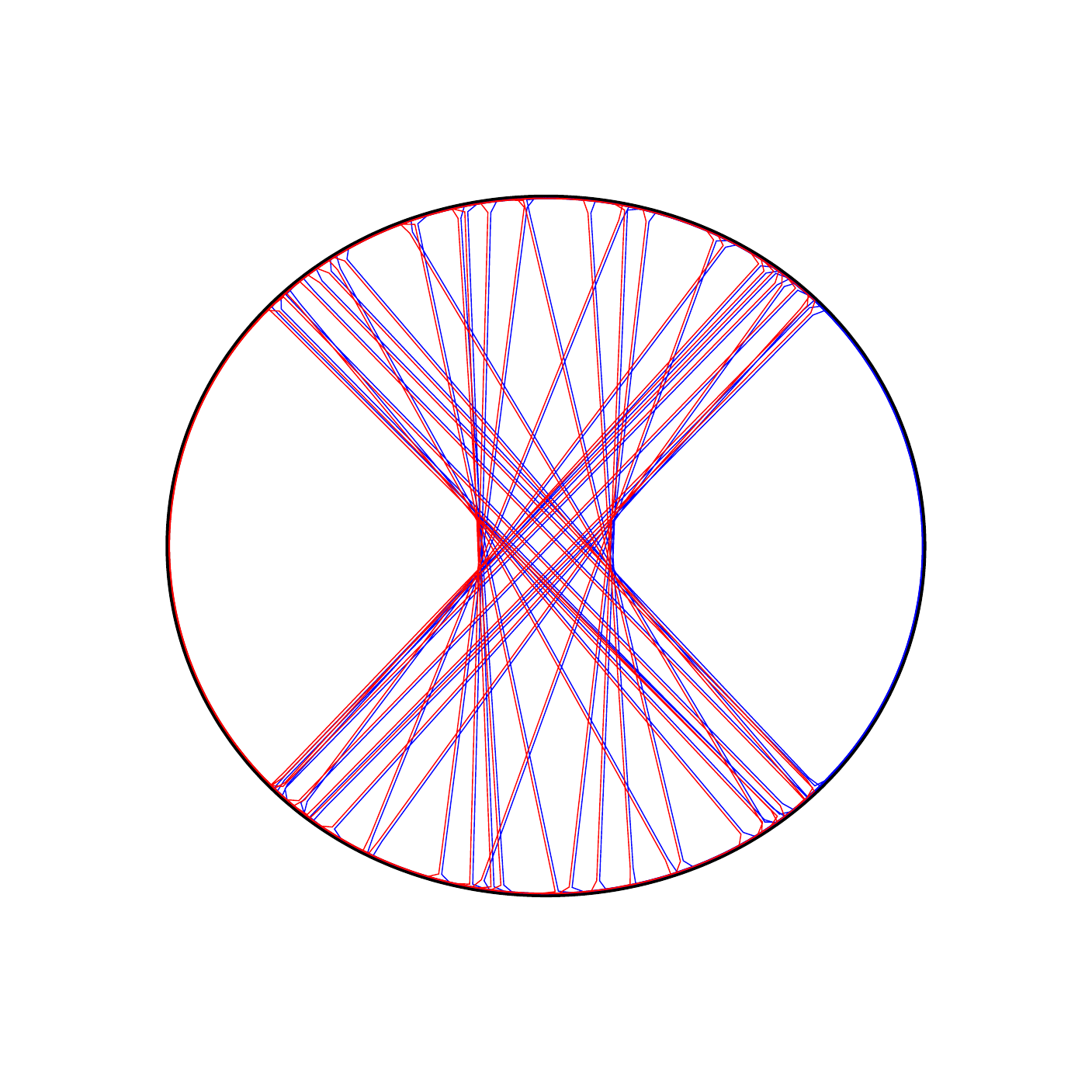} \\
    \vspace{-15mm}
            \includegraphics[width=.3\columnwidth]{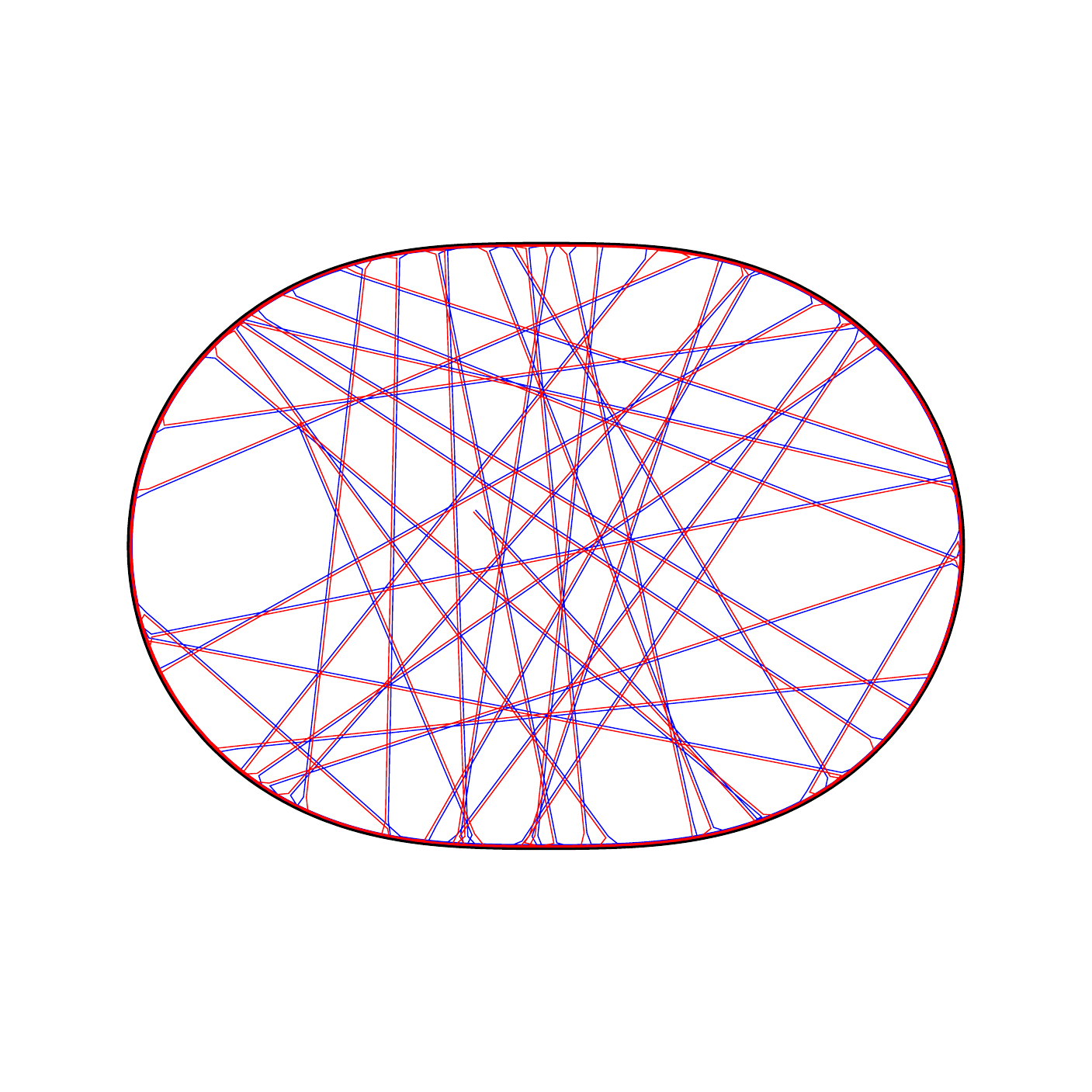} \hspace{-7mm}
                \includegraphics[width=.3\columnwidth]{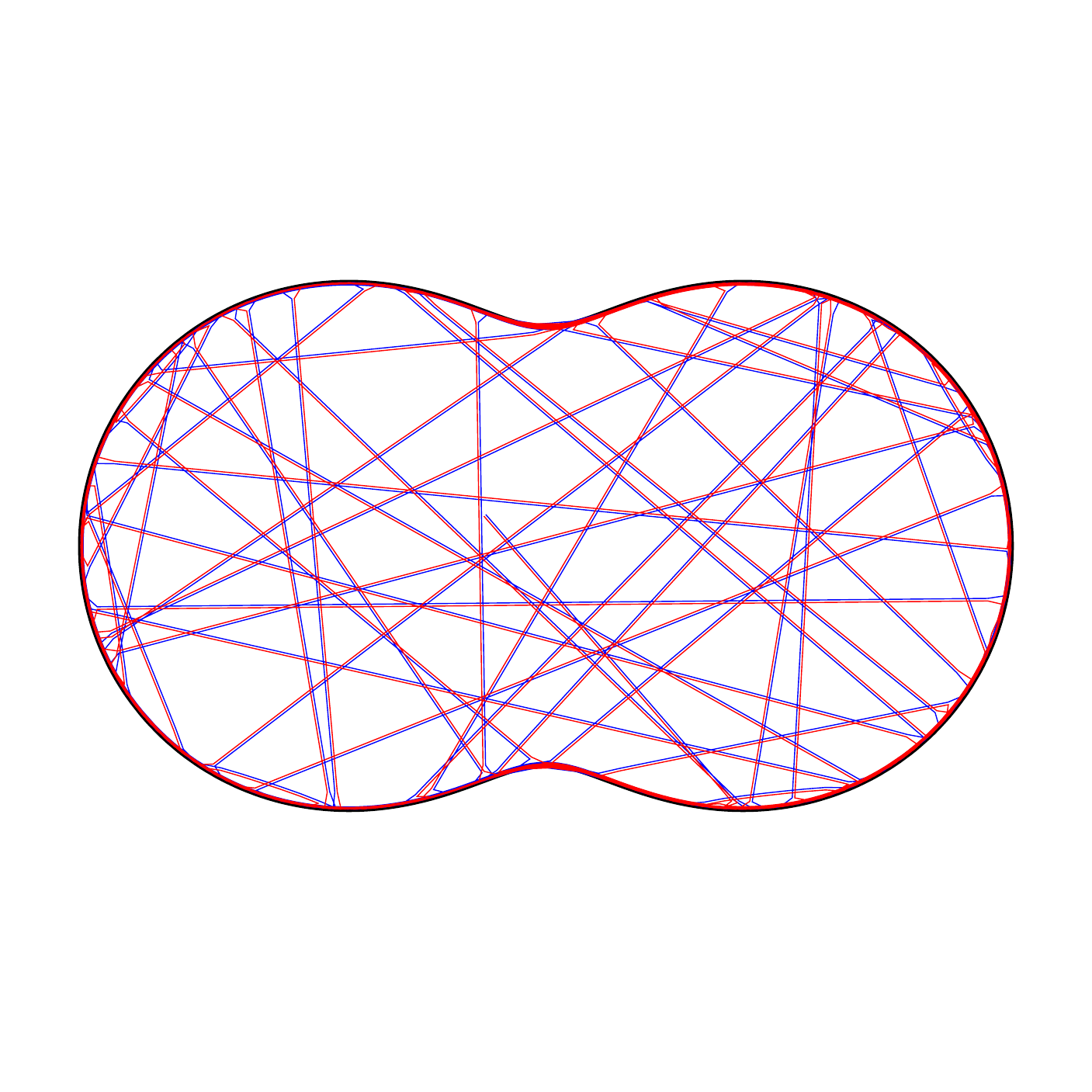} \hspace{-5mm}
                    \includegraphics[width=.3\columnwidth]{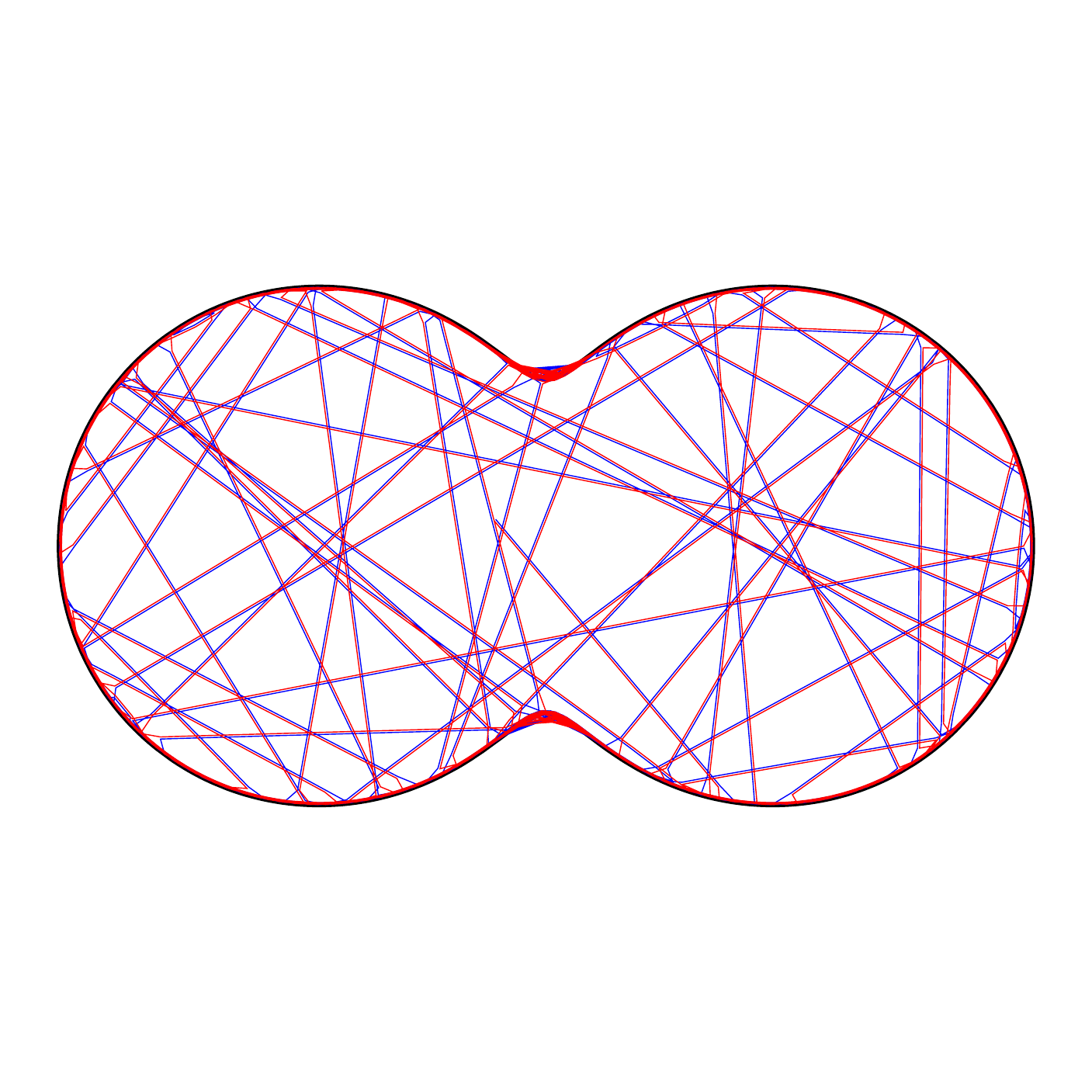} \\
                        \vspace{-10mm}
  \caption{Dipole on Neumann Oval domains with $\lambda=0,0.1,0.3,0.4,0.5,0.7$ respectively.}
\label{figOval}
\end{figure}
\vspace{2mm}

\subsubsection{Multiple interacting dipoles}

There is an important difference between the vortex billiard system as compared to other pensive billiards (such as the puck billiard).  Namely, if several classical billiards are superimposed on the same table, their interaction is almost surely absent.  On the contrary, it is fully generic that different vortex billiards on the same domain will interact, although only on the boundary of the domain, during those time intervals when the dipole is split into two monopole vortices  traveling 
separately along the boundary.  Since this boundary is one-dimensional, collisions are guaranteed and can easily happen between vortices of different pairs. That is, vortices generically exchange partners and reappear inside the domain in different vortex pairs.  See Figure \ref{figmultiple} for a  demonstration of this behavior.

For this reason, to consider more than one vortex dipole pair,   one needs to  
derive the correct limiting dynamics from the original Kirchhoff point vortex system.  However, it turns out that there is only one extra constraint in addition to the above study: if two  vortices of the same sign meet when traveling along the boundary, they simply pass by each in the limit of initial separations taken to zero.    We prove this in greater generality here, allowing the circulations of the vortices in the pair to be not equal or opposite.  The result is that there are no mergers of vortices of the same sign 
in the half-plane:

\begin{prop}\label{nomerge}
	Consider two point vortices of circulations $\ve\Gamma_1$ and $\ve\Gamma_2$ with $\Gamma_1 \neq -\Gamma_2$ in the half-plane, starting from initial positions $(x_1, \ve \hat{y}_1)$, $(x_2, \ve \hat{y}_2)$. Assume in addition that  $x_1 \neq x_2$.
	Then, there exists $\ve_0 > 0$ such that for all $0 < \ve < \ve_0$ vortices pass one another.
\end{prop}
\begin{proof}
	Introduce relative and absolute coordinates:
	\be 
	x_{\mathsf{rel}} = x_1-x_2, \quad y_{\mathsf{rel}} = y_1 - y_2,\quad
	x_{\mathsf{abs}}= \frac{\Gamma_1 x_1 + \Gamma_2 x_2}{\Gamma_1 + \Gamma_2},\quad
	y_{\mathsf{abs}}= \frac{\Gamma_1 y_1 + \Gamma_2 y_2}{\Gamma_1 + \Gamma_2}, 
	\ee
	and denote
	\be 
	\hat{y}_{\mathsf{rel}} = \hat{y}_1 - \hat{y}_2,\quad
	\hat{y}_{\mathsf{abs}}= \frac{\Gamma_1 \hat{y}_1 + \Gamma_2 \hat{y}_2}{\Gamma_1 + \Gamma_2}.
	\ee
Conservation of linear momentum $y_{\mathsf{abs}}$ and energy $H$ yields the equation of trajectory in $(x_{\mathsf{rel}}, y_{\mathsf{rel}})$ coordinates
	\begin{align}
	&\left(\hat{y}_{\mathsf{abs}} + \frac{\Gamma_2}{\Gamma_1 + \Gamma_2}\frac{y_{\mathsf{rel}}}{\ve}\right)^{\Gamma_1^2} \left(\hat{y}_{\mathsf{abs}} - \frac{\Gamma_1}{\Gamma_1 + \Gamma_2}\frac{y_{\mathsf{rel}}}{\ve}\right)^{\Gamma_2^2}\bigg[\frac{x_{\mathsf{rel}}^2 + \ve^2(2\hat{y}_{\mathsf{abs}} - \frac{\Gamma_1 - \Gamma_2}{\Gamma_1 + \Gamma_2}\frac{y_{\mathsf{rel}}}{\ve})^2}{x_{\mathsf{rel}}^2 + y_{\mathsf{rel}}^2}\bigg]^{\Gamma_1\Gamma_2} = \\
	&\qquad	\left(\hat{y}_{\mathsf{abs}} + \frac{\Gamma_2}{\Gamma_1 + \Gamma_2}\hat{y}_{\mathsf{rel}}\right)^{\Gamma_1^2} \left(\hat{y}_{\mathsf{abs}} - \frac{\Gamma_1}{\Gamma_1 + \Gamma_2}\hat{y}_{\mathsf{rel}}\right)^{\Gamma_2^2}\bigg[\frac{\hat{x}_{\mathsf{rel}}^2 + \ve^2(2\hat{y}_{\mathsf{abs}} - \frac{\Gamma_1 - \Gamma_2}{\Gamma_1 + \Gamma_2}\hat{y}_{\mathsf{rel}})^2}{\hat{x}_{\mathsf{rel}}^2 +\ve^2 \hat{y}_{\mathsf{rel}}^2}\bigg]^{\Gamma_1\Gamma_2}.
	\end{align}
	Rearranging,  the above equality gives
	\begin{align}
		x_{\mathsf{rel}}^2 &= \tfrac{y_{\mathsf{rel}}^2
			\left(\hat{y}_{\mathsf{abs}} + \frac{k}{1+k}\hat{y}_{\mathsf{rel}}\right)^{1/k}
			\left(\hat{y}_{\mathsf{abs}} - \frac{1}{1+k}\hat{y}_{\mathsf{rel}}\right)^{k} 
			\left(\frac{\hat{x}_r^2 + \ve^2(2\hat{y}_{\mathsf{abs}} - \frac{1-k}{1+k}\hat{y}_{\mathsf{rel}})^2}{\hat{x}_{\mathsf{rel}}^2 + \ve^2\hat{y}_{\mathsf{rel}}^2}\right) - \ve^2(2\hat{y}_{\mathsf{abs}} - \frac{1 - k}{1 + k}\frac{y_{\mathsf{rel}}}{\ve})^2\left(\hat{y}_{\mathsf{abs}} + \frac{k}{1+k}\frac{y_{\mathsf{rel}}}{\ve}\right)^{1/k}
			\left(\hat{y}_{\mathsf{abs}} - \frac{1}{1+k}\frac{y_{\mathsf{rel}}}{\ve}\right)^{k}}{
			\left(\hat{y}_{\mathsf{abs}} + \frac{k}{1+k}\frac{y_{\mathsf{rel}}}{\ve}\right)^{1/k}
			\left(\hat{y}_{\mathsf{abs}} - \frac{1}{1+k}\frac{y_{\mathsf{rel}}}{\ve}\right)^{k} -
			\left(\hat{y}_{\mathsf{abs}} + \frac{k}{1+k}\hat{y}_{\mathsf{rel}}\right)^{1/k}
			\left(\hat{y}_{\mathsf{abs}} - \frac{1}{1+k}\hat{y}_{\mathsf{rel}}\right)^{k} 
			\left(\frac{\hat{x}_{\mathsf{rel}}^2 + \ve^2(2\hat{y}_{\mathsf{abs}} - \frac{1-k}{1+k}\hat{y}_{\mathsf{rel}})^2}{\hat{x}_{\mathsf{rel}}^2 + \ve^2\hat{y}_{\mathsf{rel}}^2}\right)
		}\\ &=:\frac{g_{\ve}(y_{\mathsf{rel}})}{f_{\ve}(y_{\mathsf{rel}})}, \quad \text{where} \quad k := \frac{\Gamma_2}{\Gamma_1}.
	\end{align}
	We aim to show that, for sufficiently small $\ve$, $x_{\mathsf{rel}}$ is unbounded along the trajectory.  This would prove that after the two vortices meet, they do not become bound forever. For that, it is sufficient  to establish the existence of $y_{\mathsf{rel}}^{\text{crit}}$ for which 
	\be
	f_\ve(y_{\mathsf{rel}}^{\text{crit}}) = 0, \quad f'_\ve(y_{\mathsf{rel}}) \neq 0, \quad \text{and } g_\ve(y_{\mathsf{rel}}^{\text{crit}}) \neq 0.
	\ee
	To this end, consider
	\be 
	y_{\mathsf{rel}}^* =\arg\max_{y \in (-\frac{k}{1+k} \hat{y}_{\mathsf{abs}},\frac{1}{1+k} \hat{y}_{\mathsf{abs}})} \left(\hat{y}_{\mathsf{abs}} + \frac{k}{1+k}y\right)^{1/k}
	\left(\hat{y}_{\mathsf{abs}} - \frac{1}{1+k}y\right)^{k},
	\ee
and observe that for sufficiently small $\ve >0$, $f_{\ve}(y_{\mathsf{rel}}^*) > 0 $ and $ f_{\ve}(y_{\mathsf{rel}}) < 0$.
Hence, for sufficiently small $\ve$ such $y_{\mathsf{rel}}^{\text{crit}}$ satisfying $f_{\ve}(y_{\mathsf{rel}}^{\text{crit}}) = 0$ exists.  The other two conditions follow by a direct verification.
\end{proof}

\begin{cor}
Only vortices of equal strength and opposite sign can exchange their pairs. They form a multiple dipole billiard  system satisfying the fission--fusion rules. In other options the vortices of different pairs do not interact in the limit of zero initial separation.
\end{cor}

As a consequence of Proposition \ref{nomerge}, the system for multiple dipoles, strictly speaking, is no longer a billiard.
Indeed, one must keep track of the relative speeds (or, equivalently, distances to the boundary) of vortices of opposite signs
to determine the angle for their reentering the domain if  they ever meet on the boundary.  Unlike the case of one dipole, such an angle  is no longer conserved, since a vortex might meet a stranger vortex, of different speed, before its original partner, while it travels along the boundary.  One can check that the new ``fusion rule'' for 
vortices meeting along the boundary is:
\begin{enumerate}
			\item if $\Gamma_1 = -\Gamma_2$ and the ratio of their boundary velocities satisfies
			\be
			\chi^{-2}\leq \bigg|\frac{v_1}{v_2}\bigg| \leq \chi^2
			\ee
			where $\chi$ is the silver ratio, then the  monopoles form a dipole which enters the interior at angle $\theta = \theta({v_1}/{v_2})$ and moves at speed $\sqrt{-v_1v_2}$.  Here $\theta({v_1}/{v_2})$ is the explicit function derived in Proposition \ref{prop:fusion}:
			\be
				\theta(r) = \tan^{-1}\sqrt{\frac{(r - 1/r)^2}{-r^2 + 6 r - 1}}.
			\ee
			\item otherwise, monopoles continue traveling along the boundary  without interaction.
			\end{enumerate}		
It is interesting  that the addition of just one extra dipole of equal strength seems to yield a completely chaotic dynamical system.  See Figure \ref{figmultiple} for numerical simulations on a disk domain.

\begin{figure}[htb]\centering
    \includegraphics[width=.32\columnwidth]{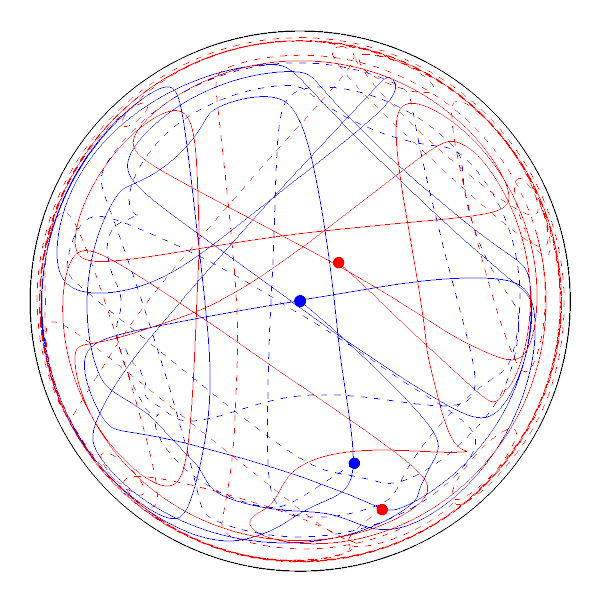} 
        \includegraphics[width=.32\columnwidth]{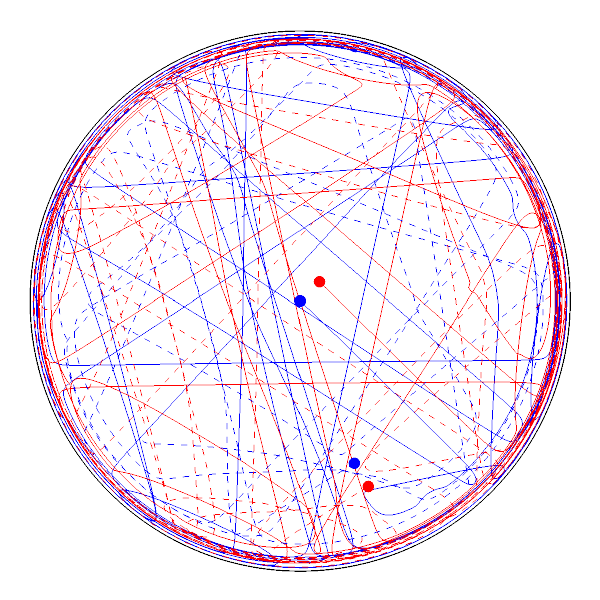} 
            \includegraphics[width=.32\columnwidth]{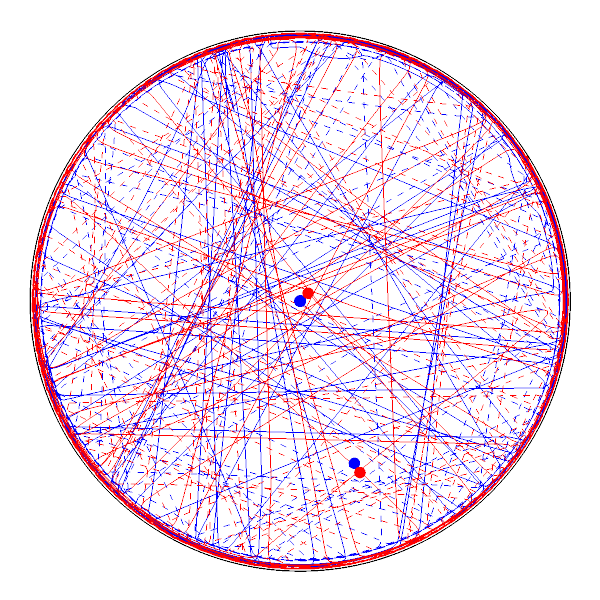} 
  \caption{Two dipole pairs on the disk. Trajectories of vortices with  positive circulation are shown in red, 
  while for vortices with   negative circulation in blue. Initially, two dipoles are formed by the solid and dashed vortices.  The separation in their initial positions is decreasing left to right.}
\label{figmultiple}
\end{figure}


\section{Pensive outer  billiards}\label{sect:outer}

\subsection{Definition of pensive outer billiards}
Recall the definition of an outer billiard, following \cite{Tab}.
Given a smooth strictly convex curve $\gamma$ in the plane, the outer 
billiard is the following map from the exterior of $\gamma$  into itself. Let $X$ be a point outside $\gamma$. 
There exist two tangent lines to $\gamma$  passing through $X$. Choose one of them, say, the right one 
from the viewpoint of $X$, and reflect $X$ in the tangency point $P$ to obtain a new point $Y=\mathsf{OB}(X)$ (see left panel of Figure \ref{figouter}).
The map $\mathsf{OB}$ is called the {\it outer billiard map} for the curve $\gamma$.

\begin{figure}[htb]\centering
	\includegraphics[width=.9\columnwidth]{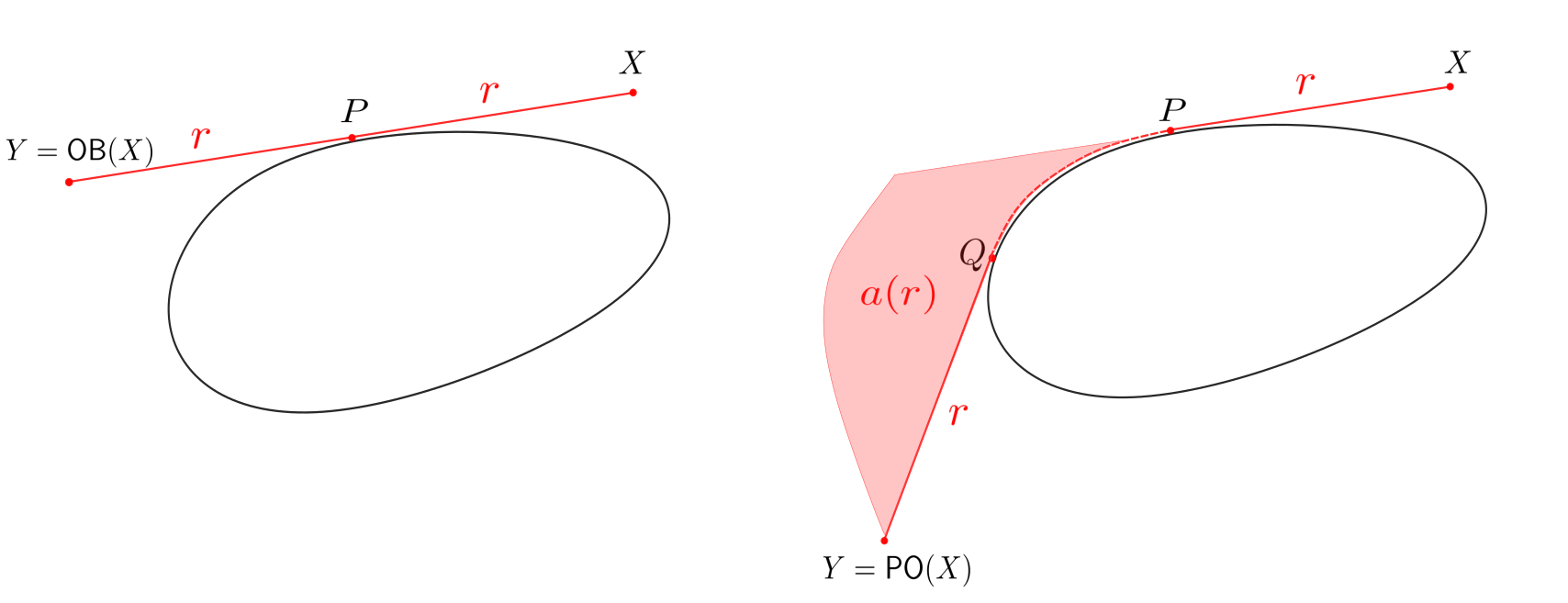} 
	\caption{Left: outer billiard. Right: pensive outer billiard.}
	\label{figouter}
\end{figure}

\begin{defi}\label{def:swept-area}
Given a strictly convex domain $D\subset \R^2$ with boundary $\gamma = \partial D$ and a delay function
$a (r), \, r>0$  a {\it pensive outer  billiard} $\mathsf{PO}:\,\R^2\setminus D \to \R^2\setminus D$ is a map
which sends a point $X$ outside $D$ to the following point $Y=\mathsf{PO}(X)$.
Take the right tangent line from $X$ to $\gamma$, suppose that the tangency point is $P\in \gamma$, and the length 
of the tangent segment is $r:=|XP|$. Now continue the motion along $\gamma$ until the tangent segment of length $r$ {\it sweeps the area} $a(r)$, and let $Q$ be the point at which it happens. After that continue moving from $Q$ along the new tangent for the distance $r$ to arrive at $Y$.
\end{defi}

Note that the definition above is a manifestation of the principle that notions related to the length in classical billiards are replaced with the area for  outer billiards.
As we show below, {\it on the plane} this definition is equivalent to the following one:

\begin{defi}{\rm (=\ref{def:swept-area}$'$)}\label{def:swept-angle}
Given a strictly convex domain $D\subset \R^2$ with boundary $\gamma = \partial D$ and a delay function
$\theta (r), \, r>0$  a {\it pensive outer  billiard} $\mathsf{PO} $ sends a point $X$ outside $D$ to the following point $Y=\mathsf{PO}(X)$. If $XP$ is the right tangent line from $X$ to $\gamma$ with the tangency point  $P\in \gamma$ and  $r:=|XP|$, we continue the motion along $\gamma$ until $Q\in \gamma$, where the tangent segment makes the {\it angle} $\theta(r)$ with the initial tangent $XP$. Then the point $Y$ is at the distance $r$ from $Q$ along this new tangent.
\end{defi}

For $ a(r)\equiv 0$  or $\theta(r)\equiv 0$ one obtains the standard outer billiard $\mathsf{OB}$.  
The equivalence of the two definitions on the plane follows from the following proposition.

\begin{prop}\label{angledef}
	The condition that the  area swept by the segment of length $r$ is equal to $a(r)$ in $\R^2$ is equivalent to the condition that the angle between tangents at $P$ and $Q$ is  $\theta(r) = 2a(r)/r^2$.
\end{prop}

Before proving this statement we introduce coordinates in which outer billiards are convenient to work with. Parametrize $\gamma$ by the angle $\alpha$ that the  tangent makes with a given (say horizontal) direction. For each point $X \in \R^2\setminus D$, let right and left tangents from $X$ to $\gamma$ have lengths $r, \bar{r}$ and meet $\gamma$ at points $\gamma(\alpha), \gamma(\bar{\alpha})$. In this way, $(\alpha,r)$ and $(\bar{\alpha}, \bar{r})$ form two coordinate systems on $\R^2\setminus D$. 
\begin{lemma}
	In these coordinate systems, standard area form on the plane is given by
	$\mu = rdr\wedge d\alpha = \bar{r}d\bar{r}\wedge d\bar{\alpha}$.
\end{lemma}
\begin{proof}
	Indeed, these are essentially polar coordinates on the plane. This also follows from a direct computation using the formulas 
	\be
	(x,y) = \gamma(\alpha) + r\frac{\gamma'(\alpha)}{|\gamma'(\alpha)|}, \quad \text{where} \quad
	\frac{\gamma'(\alpha)}{|\gamma'(\alpha)|} = (\cos\alpha,\sin\alpha),
	\ee
	and the analogous ones for $(\bar{\alpha}, \bar{r})$.
\end{proof}

Now the proof of Proposition \ref{angledef} is immediate:

\begin{proof}
	This follows from the fact that the area swept by a tangent segment of length $r$ traveling along $\gamma$ as $\alpha $ changes from $\alpha_1$ to $\alpha_2$ is  
	\be
	\int_{0}^{r}\int_{\alpha_1}^{\alpha_2}rdrd\alpha = 	\int_{0}^{r}rdr\int_{\alpha_1}^{\alpha_2}d\alpha =\frac{r^2}{2}(\alpha_2 - \alpha_1).
	\ee
\end{proof}

The main property of a pensive outer billiard, ``inherited'' from the standard outer billiard is given by the following 

\begin{theorem}
Let $D\subset \R^2$ be  convex.  The standard area 2-form $\mu$ on $\R^2$  is $\mathsf{PO}$-invariant (or, equivalently, the  pensive outer billiard $\mathsf{PO}$ is a symplectomorphism of $\R^2\setminus D$) for an arbitrary smooth function $a(r)$.
\end{theorem}

\begin{proof}

By definition of the outer billiard map, it sends a  point with coordinates $(\alpha,r)$ to the point with coordinates $(\bar{\alpha}, \bar{r}) = (\alpha, r)$. It is then immediate that $\mathsf{OB}$ is area-preserving:
\be
\mathsf{OB}^*\mu = \mathsf{OB}^*(\bar{r}d\bar{r}\wedge d\bar{\alpha}) = rdr\wedge d\alpha = \mu
\ee
We then note that by Proposition \ref{angledef}, pensive outer billiard map $\mathsf{PO}$  is a composition of  $\mathsf{OB}$ and a shift map $\mathsf{Sh}$, given in the $(\alpha, r)$ coordinates by $(\alpha,r) \mapsto (\alpha + 2a(r)/r^2, r)$. The latter map is also area-preserving:
\be
\mathsf{Sh}^*\mu = \mathsf{Sh}^*(rdr\wedge d\alpha) = rdr\wedge d(\alpha+ 2a(r)/r^2) = rdr\wedge d\alpha = \mu.
\ee 
\end{proof}

\subsection{Duality of pensive billiards on the sphere}
Similarly to classical and pensive billiards, the definition of the pensive outer  billiard can be generalized to non-flat domains. 
The following consideration of the pensive outer  billiard on a spherical domain justifies Definition  \ref{def:swept-area} with the swept area, as opposed to Definition \ref{def:swept-angle} valid only in the flat case: on the sphere these two definitions are not equivalent!
\begin{theorem}\label{duality}
	The pensive billiard with delay function $\tilde{\ell}(\theta)$ and pensive outer  billiard with delay function $a(\theta)$ are projectively dual on the unit sphere, provided that 
	\be\label{dualityrelation}
	a(\theta) = \tilde{\ell}(\theta)(1 - \cos\theta).
	\ee
\end{theorem}
	\begin{figure}[htb]\centering
		\includegraphics[width=.5\columnwidth]{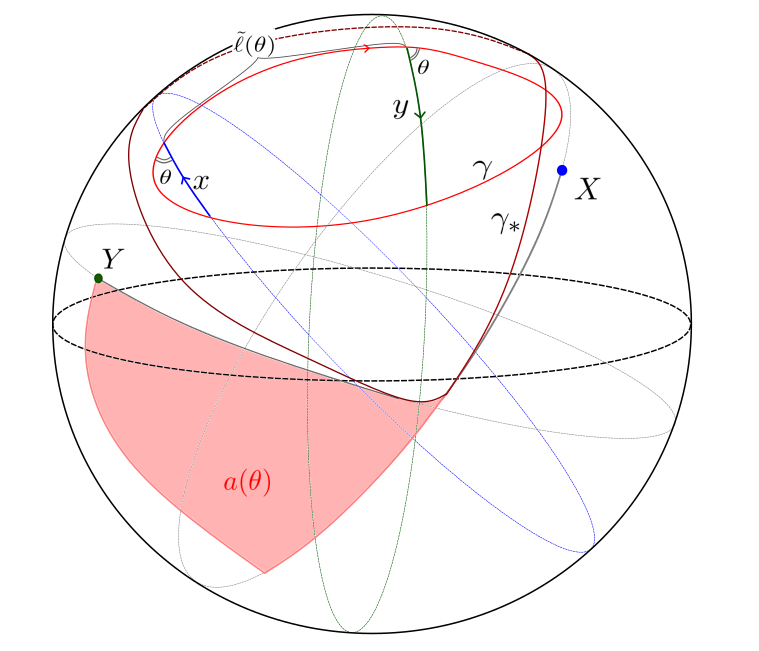} 
		\caption{Duality of pensive billiards and pensive outer billiards on the sphere. Pensive billard map with respect to red curve $\gamma$ sends great circle $x$ to great circle $y$. Pensive outer billiard map with respect to dual curve $\gamma_*$ sends point $X$ dual to $x$ to point $Y$ dual to $y$. Moreover, functions $\tilde{\ell}$ and $a$ are related by Equation \eqref{dualityrelation}}
		\label{figduality}
	\end{figure}
	This theorem mimics the  one for the classical and outer billiards on the sphere, see \cite{Tab}.
	The projective duality in the sphere interchanges points (poles) and the corresponding oriented great circles (equators), analogs of lines on the sphere. An oriented tangent line (great circle) to a curve $\gamma$  is sent to a point of the dual curve $\gamma_*$, which by definition  consists of poles of the tangent great circles to $\gamma$. 	
	An immediate corollary of this duality  is that the pensive outer billiard  on the sphere preserves the spherical areas.

	Compared to the classical case, we need the following additional  geometric result, which can be regarded as a generalization of the Archimedes theorem (as we discuss below).
	
	\begin{prop}\label{geolem}
		Let $\sigma(s)$ be any curve in $S^2 \subset \R^3$ (not necessarily parametrized by arclength). The area swept by the segment of angular length $\theta$ tangent to $\sigma$ at $\sigma(s)$ as $s$ varies between $s_1$ to $s_2$ is 
		\be
		(1 - \cos\theta)\int_{s_1}^{s_2}\frac{|\sigma''(s) \cdot (\sigma(s) \times \sigma'(s)))|}{|\sigma'(s)|^2}ds.
		\ee
	\end{prop} 
\begin{proof}
	The region swept by the tangent segment of length $\theta$ is parametrized by 
	\be
	r(s,\beta) = \cos\beta\sigma(s) + \sin\beta\frac{\sigma'(s)}{|\sigma'(s)|}, \qquad (s, \beta) \in [s_1, s_2]\times[0, \theta].
	\ee
	The area is consequently computed as 
	\be
	\begin{aligned}
	&\int_0^\theta\int_{s_1}^{s_2}\bigg|\frac{\partial r}{\partial s}\times \frac{\partial r}{\partial \beta}\bigg| \,ds\,d\beta = \int_0^\theta\int_{s_1}^{s_2}\bigg|r \cdot \bigg(\frac{\partial r}{\partial s}\times \frac{\partial r}{\partial \beta}\bigg)\bigg| \,ds\,d\beta \\
	&= \int_0^\theta\int_{s_1}^{s_2}\sin\beta\frac{|\sigma''(s) \cdot (\sigma(s) \times\sigma'(s))|}{|\sigma'(s)|^2}\,ds\,d\beta 
	=(1 -\cos \theta)\int_{s_1}^{s_2}\frac{|\sigma''(s) \cdot (\sigma(s) \times \sigma'(s)))|}{|\sigma'(s)|^2}ds,
	\end{aligned}
\ee	
where we used that $\sigma(s)\cdot\sigma'(s) = 0$ together with the symmetries of the triple product.
\end{proof}
\begin{proof}[Proof of Theorem]
	The (pensive) billiard map acts on oriented curves: an incidence ray is mapped  to the reflected one. 
	For the classical billiard, at the moment of reflection in $\gamma$ the incidence ray $a$, the tangent $p$ to $\gamma$
	and the reflected ray $b$ pass through the same point and make equal angles. On the sphere this translates to the corresponding three poles $A, P,$ and $B$, lying on the same line (great circle) and making distances equal to the angle of incidence, $|AP|=|{PB}| = \theta$. The latter is the definition of the outer billiard. 
	
	For the pensive billiards, before reflecting the point slides along 
	$\gamma$ for the distance $\tilde{\ell}(\theta)$ depending on the incidence angle $\theta$ and then reflected from the tangent $\tilde p$. We shall see that the area swept by the corresponding segment moving along the dual curve is equal to 
	\be
	\tilde{\ell}(\theta)(1 - \cos\theta).
	\ee 
	Given an arclength parametrization $\gamma(s)$ of $\gamma$, dual curve $\gamma_*$ can be parametrized by 
	\be
	\gamma_*(s) = \gamma(s)\times \gamma'(s).
	\ee
	Hence, the statement will follow from Proposition \ref{geolem} applied to $\sigma = \gamma_*$ if we prove that 
	\be
	\frac{|\gamma_*''(s) \cdot (\gamma_*(s) \times \gamma_*'(s)))|}{|\gamma_*'(s)|^2} = 1.
	\ee
	This is indeed the case. 
	Expanding $\gamma''(s)$ in the ortonormal basis $(\gamma(s), \gamma'(s), \gamma(s)\times\gamma'(s))$ we find 
	\be
	\gamma''(s) = -\gamma(s) + (\gamma''(s)\cdot(\gamma(s)\times\gamma'(s)))\gamma(s)\times\gamma'(s).
	\ee
	Hence,
	\be
	|\gamma_*'(s)|^2 = |\gamma(s)\times\gamma''(s)|^2| = |(\gamma''(s)\cdot(\gamma(s)\times\gamma'(s))|^2.
	\ee
	On the other hand, 
		\be
	\begin{aligned}
	|\gamma_*''(s) \cdot (\gamma_*(s) \times \gamma_*'(s)))| &= |\gamma_*'(s) \cdot (\gamma_*(s) \times \gamma_*''(s)))| = 
	|(\gamma(s)\times\gamma''(s)) \cdot (\gamma_*(s) \times \gamma_*''(s)))| =\\
	&=|\gamma''(s)\cdot(\gamma(s)\times\gamma'(s))| |\gamma'(s) \cdot (\gamma_*(s) \times \gamma_*''(s)))|\\
	&=|\gamma''(s)\cdot(\gamma(s)\times\gamma'(s))| |\gamma_*''(s) \cdot \gamma(s)|\\
	&=|\gamma''(s)\cdot(\gamma(s)\times\gamma'(s))| |(\gamma \times \gamma'''(s) + \gamma'(s) \times \gamma''(s)) \cdot \gamma(s)|\\
	&=|\gamma''(s)\cdot(\gamma(s)\times\gamma'(s))||(\gamma'(s) \times \gamma''(s)) \cdot \gamma(s)|\\
	&=|\gamma''(s)\cdot(\gamma(s)\times\gamma'(s))|^2.
	\end{aligned}
	\ee
	This concludes our proof.
\end{proof}
Note that Proposition \ref{geolem} includes a theorem due to Archimedes as a special case:
\begin{cor}[Archimedes Theorem]
	Enclose the unit sphere by a cylinder of diameter 2 and height 2. The projection of the sphere onto this cylinder preserves area. 
\end{cor}
\begin{proof}[Proof of Corollary]
	It is enough to prove the theorem for a region $R$ bounded by two parallel planes $\{z = h_1\}$ and $\{z = h_2\}$, where we can assume $0<h_2 < h_1$. We note that this region is given by swiping a tangent segment along the circle lying in the first plane; one can see that the length $\theta$ of the segment needs to satisfy $h_1(1 - \cos\theta) = h_1 - h_2$. Applying lemma to this segment, note that by  symmetry of the triple product, the integrand is given by the projection of $\gamma(s)$ onto $\gamma'(s)\times \gamma''(s)$, which is here just a projection of $\gamma(s)$ onto the vertical direction, i.e. $h_1$.
	
	Consequently, area of $R$ is given by 
	\be
			(1 - \cos\theta)\int_{s_1}^{s_2}\frac{|\sigma''(s) \cdot (\sigma(s) \times \sigma'(s)))|}{|\sigma'(s)|^2}ds = \frac{h_1 - h_2}{h_1}\int_{0}^{2\pi}h_1 ds = 2\pi(h_1 - h_2),
	\ee
	which is the same as the area of the projection of $R$ onto cylinder.
\end{proof}

The above duality of two types of pensive billiards allows one to make hydrodynamical meaning
for pensive outer billiards  on the sphere: they are dual to systems of vortex dipoles in spherical regions in the limit of zero separation.

\medskip

We note  that, since many facts in dynamics often turn out to be simpler to prove for outer billiards than for classical ones, it would be interesting 
to extend other discussed results, including the twist property, generating functions, the golden and silver ratios, periodic trajectories, etc. to the domain of pensive outer billiards.  One can consider also pensive outer billiards on other surfaces with constant curvature, for example the hyperbolic plane.  Various features, such as area preservation of the billiard map, can be extended to that case.

In conclusion, there are many interesting questions worth investigating for the pensive billiard system which we have not studied. For example
\begin{itemize}
\item The pensive billiard map has a variational formulation but, in general, is not a twist map. What can be said (using Morse theory) about the periodic orbits in this case?
\item We prove very little about the existence (and non-existence) of caustics. Does KAM theory applies, is there a version of Lazutkin’s theorem in this setting? What can be said towards Mather’s non-existence criterion?
First results in that direction for puck billiards were recently obtained in \cite{Barbieri}.
\item The billiard system inside ellipses is integrable. For which combinations of the delay functions and shapes of  billiard tables will the pensive billiards be integrable? 
\item What is the multidimensional version of pensive billiards?
\end{itemize}

 \subsection*{Acknowledgments}  
 We are grateful to an anonymous referee for highlighting the concluding questions.
TDD  is also grateful  for support from the Alfred P. Sloan foundation. DG is also grateful for support from Simons Foundation International, LTD.  BK is  indebted to Sergei Tabachnikov for fruitful discussions, as well as to the Simons Center for Geometry and Physics (SCGP) and to the Institut des Hautes \'Etudes Scientifiques (IHES) for their stimulating research environment and support. 
 \smallskip
 
Conflicts of interest: none.

 \subsection*{Financial Support}  
  TDD and DG were partially supported by the NSF DMS-2106233 grant and NSF CAREER award \#2235395.    BK  was partially supported by an NSERC Discovery Grant RGPIN-2025-06427.

\bibliographystyle{acm}
\bibliography{biblio-dipole}

\begin{thebibliography}{10}

\bibitem{ashbee2013generalized}
{\sc Ashbee, T., Esler, J., and McDonald, N.}
\newblock Generalized hamiltonian point vortex dynamics on arbitrary domains
  using the method of fundamental solutions.
\newblock {\em Journal of Computational Physics 246\/} (2013), 289--303.

\bibitem{ashbee2014dynamics}
{\sc Ashbee, T.~L.}
\newblock {\em Dynamics and statistical mechanics of point vortices in bounded
  domains}.
\newblock PhD thesis, UCL (University College London), 2014.

\bibitem{Barbieri}
{\sc Barbieri, S., and Clarke, A.}
\newblock Existence and nonexistence of invariant curves of coin billiards.
\newblock {\em arXiv:2411.13214\/} (2024).

\bibitem{Bialy}
{\sc Bialy, M., Fierobe, C., Glutsyuk, A., Levi, M., Plakhov, A., and
  Tabachnikov, S.}
\newblock Open problems on billiards and geometric optics.
\newblock {\em arXiv:2110.10750\/} (2021).

\bibitem{blaschke1916kreis}
{\sc Blaschke, W.}
\newblock {\em Kreis und Kugel}.
\newblock De Gruyter, 1916.

\bibitem{boatto2015}
{\sc Boatto, S., and Koiller, J.}
\newblock Vortices on closed surfaces.
\newblock {\em Fields Institute Communications 73\/} (2015), 185--237.

\bibitem{DragovicRadnovich}
{\sc Dragović, V., and Radnović, M.}
\newblock Magic billiards: the case of elliptical boundaries.
\newblock {\em arXiv:2409.03158\/} (2025).

\bibitem{drivas2023singularity}
{\sc Drivas, T.~D., and Elgindi, T.~M.}
\newblock Singularity formation in the incompressible {E}uler equation in
  finite and infinite time.
\newblock {\em EMS Surveys in Mathematical Sciences 10}, 1 (2023), 1--100.

\bibitem{DGK24}
{\sc Drivas, T.~D., Glukhovskiy, D., and Khesin, B.}
\newblock Singular vortex pairs follow magnetic geodesics.
\newblock {\em International Mathematics Research Notices\/} (2024), rnae106.

\bibitem{DIN2023feynman}
{\sc Drivas, T.~D., Iyer, S., and Nguyen, T.~T.}
\newblock The {F}eynman--{L}agerstrom criterion for boundary layers.
\newblock {\em Archive for Rational Mechanics and Analysis 248}, 3 (2024), 55.

\bibitem{flucher1997vortex}
{\sc Flucher, M., and Gustafsson, B.}
\newblock Vortex motion in two-dimensional hydromechanics.
\newblock {\em Preprint in TRITA-MAT-1997-MA-02\/} (1997).

\bibitem{FomenkoVZ}
{\sc Fomenko, A.~T., Vedyushkina, V.~V., and Zav'yalov, V.~N.}
\newblock Liouville foliations of topological billiards with slipping.
\newblock {\em Russ. J. Math. Phys. 28}, 1 (2021), 37--55.

\bibitem{geshev2018motion}
{\sc Geshev, P., and Chernykh, A.}
\newblock The motion of vortices in a two-dimensional bounded region.
\newblock {\em Thermophysics and Aeromechanics 25\/} (2018), 809--822.

\bibitem{grotta2024interplay}
{\sc Grotta-Ragazzo, C., Gustafsson, B., and Koiller, J.}
\newblock On the interplay between vortices and harmonic flows: Hodge
  decomposition of {E}uler’s equations in 2d.
\newblock {\em Regular and Chaotic Dynamics 29}, 2 (2024), 241--303.

\bibitem{gustafsson1979}
{\sc Gustafsson, B.}
\newblock On the motion of a vortex in two-dimensional flow of an ideal fluid
  in simply and multiply connected domains.
\newblock {\em preprint TRITA-MAT-1979-7, see also arXiv:2405.19215\/} (1979),
  1--112.

\bibitem{katok-iet}
{\sc Katok, A.}
\newblock Interval exchange transformations and some special flows are not
  mixing.
\newblock {\em Israel Journal of Mathematics 35}, 4 (Dec. 1980), 301–310.

\bibitem{Katok_Hasselblatt}
{\sc Katok, A., and Hasselblatt, B.}
\newblock {\em Introduction to the Modern Theory of Dynamical Systems}.
\newblock Encyclopedia of Mathematics and its Applications. Cambridge
  University Press, 1995.

\bibitem{Katok_1986}
{\sc Katok, A., Strelcyn, J.-M., Ledrappier, F., and Przytycki, F.}
\newblock {\em Invariant Manifolds, Entropy and Billiards; Smooth Maps with
  Singularities}.
\newblock Springer Berlin Heidelberg, 1986.

\bibitem{KhesinWang}
{\sc Khesin, B., and Wang, H.}
\newblock The golden ratio and hydrodynamics.
\newblock {\em The Mathematical Intelligencer 44}, 1 (Aug. 2021), 22–27.

\bibitem{kimura}
{\sc Kimura, Y.}
\newblock Vortex motion on surfaces with constant curvature.
\newblock {\em Proceedings of the Royal Society of London. Series A 455}, 1981
  (1999), 245--259.

\bibitem{love1893motion}
{\sc Love, A.}
\newblock On the motion of paired vortices with a common axis.
\newblock {\em Proceedings of the London Mathematical Society 1}, 1 (1893),
  185--194.

\bibitem{MarchioroPulvirenti}
{\sc Marchioro, C., and Pulvirenti, M.}
\newblock {\em Mathematical Theory of Incompressible Nonviscous Fluids}.
\newblock 1993.

\bibitem{mavroyiakoumou2020collinear}
{\sc Mavroyiakoumou, C., and Berkshire, F.}
\newblock Collinear interaction of vortex pairs with different
  strengths—criteria for leapfrogging.
\newblock {\em Physics of Fluids 32}, 2 (2020).

\bibitem{ModinViviani}
{\sc Modin, K., and Viviani, M.}
\newblock Integrability of point-vortex dynamics via symplectic reduction: a
  survey.
\newblock {\em Arnold Math. J. 7}, 3 (2021), 357--385.

\bibitem{newton2001n}
{\sc Newton, P.}
\newblock {\em The {N}-Vortex Problem: Analytical Techniques}, vol.~145 of {\em
  Applied Mathematical Sciences}.
\newblock 2001.

\bibitem{ryabov2019bifurcation}
{\sc Ryabov, P.~E., and Shadrin, A.~A.}
\newblock Bifurcation diagram of one generalized integrable model of vortex
  dynamics.
\newblock {\em Regular and Chaotic Dynamics 24\/} (2019), 418--431.

\bibitem{ryabov2021dynamics}
{\sc Ryabov, P.~E., Sokolov, S.~V., and Palshin, G.~P.}
\newblock Dynamics of a few vortices in an ideal fluid and in a
  {B}ose--{E}instein condensate.
\newblock {\em arXiv preprint arXiv:2103.11667\/} (2021).

\bibitem{sokolov2017bifurcation}
{\sc Sokolov, S.~V., and Ryabov, P.~E.}
\newblock Bifurcation analysis of the dynamics of two vortices in a
  {B}ose--{E}instein condensate. the case of intensities of opposite signs.
\newblock {\em Regular and Chaotic Dynamics 22\/} (2017), 976--995.

\bibitem{Tab}
{\sc Tabachnikov, S.}
\newblock {\em Geometry and {B}illiards}, vol.~30.
\newblock American Mathematical Soc., 2005.

\bibitem{yang}
{\sc Yang, C.}
\newblock Vortex motion of the {E}uler and lake equations.
\newblock {\em Journal of Nonlinear Science 31}, 3 (Apr 2021).

\bibitem{zemlyakov-katok}
{\sc Zemlyakov, A.~N., and Katok, A.~B.}
\newblock Topological transitivity of billiards in polygons.
\newblock {\em Mathematical Notes of the Academy of Sciences of the USSR 18}, 2
  (Aug. 1975), 760–764.

\end{thebibliography}

\end{document}